\documentclass[11pt]{amsart}
\usepackage{amsmath}
\usepackage{amssymb}
\usepackage{amsfonts}
\usepackage{amsthm}
\usepackage{mathabx}
\usepackage{mathrsfs}

\newcommand*\norm[1]{ \left|\left| #1 \right|\right| }

\newcommand*\Z{ \mathbb{Z} }
\newcommand*\Q{ \mathbb{Q} }
\newcommand*\R{ \mathbb{R} }
\newcommand*\C{ \mathbb{C} }
\newcommand*\N{ \mathbb{N} }

\newcommand\set[1]{\left\{ #1 \right\}}

\newcommand\INNERPROD[2]{\left\langle #1, #2 \right\rangle}
\newcommand*\FLOOR[1]{ \left\lfloor #1 \right\rfloor }

\DeclareMathOperator{\supp}{supp}

\DeclareMathOperator{\IM}{Im}

\theoremstyle{theorem}
\newtheorem{mydef}{Definition}[section]

\newtheorem{Remark}{Remark}

\theoremstyle{theorem}

\newtheorem{mythm}{Theorem}[section]

\newtheorem{mylemma}{Lemma}[section]

\newtheorem{mycor}{Corollary}[section]

\newtheorem{myq}{Question}

\title{Fine dimensional properties of spectral measures}
\author{Michael Landrigan \and Matthew Powell}
\date{\today}

\begin{document}
\maketitle

\begin{abstract}
Operators with zero dimensional spectral measures appear naturally in the theory of ergodic Schr\"odinger operators. We develop the concept of a complete family of Hausdorff measure functions in order to analyze and distinguish between these measures with any desired precision.
We prove that the dimension of spectral measures of half-line operators with positive upper Lyapunov exponent are at most logarithmic for every possible boundary phase.
We show that this is sharp by 
constructing an explicit operator whose spectral measure obtains this dimension. 
We also extend and improve some basic results from the theory of rank one perturbations and quantum dynamics to encompass generalized Hausdorff dimensions.

\end{abstract}

\section{Introduction}

The classification of measures using the classical power-law Hausdorff measures and dimensions has found many applications within spectral theory (\cite{DamTch1, DelRioJitLastSim, HLQZ2019, JitomirskayaKrasovsky, LanaLast, LanaLast2, jlt, LanaMavi1, Tch2, TCHPotentials} and others). While this classification theory has been very useful in many situations, notably when the Hausdorff dimension is positive, it has not been general enough to understand the differences between zero-dimensional spectra. This has been explored in recent papers by Mavi \cite{Mavi1}, who studied logarithmic dimension bounds for the disordered Holstein model, and Avila, Last, Shamis, and Zhou  \cite{ALSZ21}, who studied the modulus of continuity of the integrated density of states for the almost Mathieu operator using a logarithmic dimension. 

Our primary purpose in this paper is to develop general tools to study these and even finer spectral questions. Explicitly, we consider different kinds of singular-continuous measures based on a more general notion of Hausdorff measure and dimension. The relevant definitions are discussed in section \ref{section:IntroDefs}. 

One of the advantages of our approach is that it allows us to distinguish between measures that are classically termed "zero-dimensional". These broad questions are very relevant to the study of quantum dynamics, where the fractal properties of spectral measures are usually connected with anomalous transport properties (e.g. \cite{BarbarouxGerminet, Last}), while "zero-dimensional" spectral measures naturally occur when studying ergodic Schr\"odinger operators with positive Lyapunov exponent (see e.g. \cite{LanaHan1, LanaLast, LanaLast2, LanaMavi1, SimonDim}). In particular, a general result due to Simon \cite{SimonDim} says that the spectral measures, $\mu_\theta,$ associated to an ergodic family of Schr\"odinger operators, $H_\theta,$ on $l^2(\Z)$ with positive Lyapunov exponent are supported on sets of logarithmic capacity 0. This implies that the spectral measures are zero-dimensional. 

Results pertaining to dynamics have always been closely tied to the dimensional characteristics of spectral measures, so a finer distinction between dimensions should also provide additional tools to strengthen dynamics results. Additionally, recent work by Jitomirskaya and Liu \cite{JitoLiuRef}  leads us to expect that the existence of phase resonances in quasiperiodic models implies very deeply zero-dimensional spectral measures whose dimensional properties cannot be well understood with classical notions, or even the $\log$-dimension which was developed by one of the authors in his thesis \cite{LandriganThesis} and has been studied in recent papers \cite{ALSZ21, Mavi1}.

The generality of our analysis is only possible because of our development and exploration of a {\it complete family of Hausdorff dimension functions} (Definition \ref{CompFam}). 
In particular, a key technical component of our theory is Theorem \ref{SUBHDIM}, which only discusses a single Hausdorff measure. This theorem only becomes useful in practice once we restrict our attention to a suitable collection of Hausdorff measures, rather than all possible Hausdorff measures, since it is not possible to compare all Hausdorff measures to one another. This has been done in the past by considering powers of suitable gauge functions, such as $\log(1/t)^{-1},$ but this is not suitable for our fine analysis; for example, the dimension of a set with respect to the two families $\mathcal{F}_1 = \set{\log(1/t)^{-\alpha}: \alpha > 0}$ and $\mathcal{F}_2 = \set{(\log(1/t)\log\log(1/t))^{-\alpha}: \alpha > 0}$ will always coincide. Determining which of the two is "closer" to the actual dimension requires a more general type of family.

While we are interested in an abstract theory, we are especially motivated by two phenomena that we know yield zero-dimensional spectra:  
\begin{enumerate}
\item Schr\"odinger operators with positive Lyapunov exponent;
\item Local perturbations of systems with exponentially localized eigenfunctions.
\end{enumerate}

We also extend the theory of quantum dynamics to our more general setting, but we save applications to our sequel.
We will also save the study of operators with positive Lyapunov exponent for our sequel, and here we will instead consider the case where the upper Lyapunov exponent is positive (see \eqref{Lupper} for the relevant definition).

We are motivated by one particular question when considering the regime of positive upper Lyapunov exponent:
\begin{myq}
Does positive upper Lyapunov exponent imply an upper bound on how singular the spectral measure must be?
\end{myq}

In 1999, Jitomirskaya and Last  \cite{LanaLast} proved that spectral measures for half-line Schr\"odinger operators with phase boundary condition $\theta$ and positive upper Lyapunov exponent must be zero Hausdorff dimensional (in the classical sense of Hausdorff dimension) for every $\theta.$ 

In 2001, one of the authors \cite{LandriganThesis} introduced the notion of logarithmic dimension, and proved that spectral measures for half-line Schr\"odinger operators with phase boundary condition $\theta$ and positive upper Lyapunov exponent must have logarithmic dimension at most 1 for every $\theta.$ The results of the thesis \cite{LandriganThesis} were never published previously and are incorporated here. In our current framework, the logarithmic dimension coincides with $\dim_\mathcal{F}$ when $\mathcal{F} = \set{\log(1/t)^{-\alpha}: 0 < \alpha < \infty}.$

In 2007, Simon \cite{SimonDim} proved that, given a family of ergodic Schr\"odinger operators $H_\theta$ with positive Lyapunov exponent, the spectral measure $\mu_\theta$ must be supported on a set with zero logarithmic capacity for a.e. $\theta.$ The ergodicity and positive Lyapunov exponent conditions in \cite{SimonDim} are more restrictive than our requiring positive upper Lyapunov exponent. Hence, Simon's result still leaves us with the question of what happens in the setting of positive upper Lyapunov exponent, as well as what happens on the excluded Lebesgue null set.

These results lead us to consider the following, more refined question:
\begin{myq}
Does positive upper Lyapunov exponent imply that the spectral measure is always singular with respect to the $\ln(1/t)^{-1}$-Hausdorff measure?
\end{myq}
In this paper, we answer this in two ways for half-line operators with phase boundary condition $\theta.$ We prove that 
$\dim_\mathcal{F}(\mu_\theta) = \log(1/t)^{-1},$ when considering the family $\mathcal{F} = \set{\log(1/t)^{-\alpha}: 0 < \alpha < \infty},$ (Theorem \ref{SPARSEBARTHM}) and that more generally, the spectral measure must be at least $(\log(1/t)\log\log(1/t)^2)^{-1}$-singular (Theorem \ref{UPPERLYAP}), where both results hold for every phase $\theta.$ Furthermore, we construct half-line operators with phase boundary condition $\theta$ and positive upper Lyapunov exponent for every $\theta$ such that $\dim_\mathcal{F}(\mu_\theta) = \log(1/t)^{-1}$ for Lebesgue a.e. $\theta$ and any complete family of Hausdorff dimension functions, $\mathcal{F},$ such that $\log(1/t)^{-1} \in \mathcal{F}$ (Theorem \ref{SPARSEBARTHM}). These show that the ideal bound lies somewhere between $\log(1/t)^{-1}$ and $(\log(1/t)\log\log(1/t)^2)^{-1}.$ These are the main results of our paper.

This improves the result from \cite{LanaLast} and shows that the bound is sharp for $\log$-dimension, but not necessarily for our more refined notion of dimension. In comparing our results to Simon's \cite{SimonDim}, we are drawn to two major differences: (1) we do not assume ergodicity and (2) we do not exclude a Lebesgue null set. 

The salient point here is that we arrive at a result similar to that in \cite{SimonDim} without the assumption of ergodicity or positivity of the Lyapunov exponent; we just need positive upper Lyapunov exponent. It is possible to view an ergodic family of operators as similar to a family of operators with a phase boundary condition. This surface analogy would lead us to believe that the result in \cite{SimonDim} should be the same as the result in our situation, but this is not the case. Moreover, if there were an analogy between the ergodic parameter and the phase boundary parameter, then Simon's result would lead one to believe that a Lebesgue null set of phases needs to be excluded; our result shows that this is not true. We are able to obtain a logarithmic bound for all boundary phases.

A truer comparison between \cite{SimonDim} and the analysis that can be performed using a complete family of Hausdorff measure functions can be found in upcoming work \cite{MP2021}.

We are also interested in local perturbations of systems with exponentially localized eigenfunctions and the following question:
\begin{myq}
Suppose $A: l^2(\Z^\nu) \to l^2(\Z^\nu)$ is self-adjoint with semi-uniformly localized eigenfunctions, and let $\mu = \mu_0$ be the spectral measure for $\delta_0.$ If $A_\lambda = A + \lambda\INNERPROD{\delta_0}{\cdot}\delta_0$ is a rank one perturbation at the origin, and if $\mu_\lambda$ is the spectral measure for $\delta_0$ associated to $A_\lambda,$ is there an upper bound on how singular $\mu_\lambda$ is?
\end{myq}

In 1996, del Rio, Jitomirskaya, Last, and Simon proved that $\mu_\lambda$ must be zero Hausdorff dimensional (in the classical sense of Hausdorff dimension) for every $\lambda.$ 

We refine this answer in the following way (Theorem \ref{MySULERes}): not only are the spectral measures zero-dimensional, but the spectral measures $\mu_\lambda$ are in fact $(\log(1/t))^{-\nu-\epsilon}$-singular for every $\lambda$ and $\epsilon > 0.$

We also rigorously extend the quantum dynamic theory of Last from the power-law setting to the general Hausdorff dimension setting (section \ref{section:QuantDynam}). A similar result extending quantum dynamics appears in Mavi \cite{Mavi1}. We believe that these results, especially Theorem \ref{POSBDTHM}, can lead to a strengthening of existing dynamics results, as well as very fine results for quasiperiodic models (c.f. \cite{LanaHan1, JitomirskayaLiuMaryland, jlt, LanaMavi1}).

The starting point for our analysis is the decomposition theory of Rogers and Taylor \cite{RT1, RT2}. Classically, any $\sigma$-finite measure can be decomposed into pure point, singular continuous, and absolutely continuous parts via the Lebesgue decomposition theorem; Rogers and Taylor took this further and decomposed the singular continuous part into measures that are singular or continuous with respect to the power-law Hausdorff measures. 

Briefly, a measure $\mu$ is said to have exact power-law dimension $\alpha \in [0,1]$ if and only if $\mu(E) = 0$ for every set $S$ with power-law Hausdorff dimension $\beta < \alpha$ and if $\mu$ is supported on a set a power-law Hausdorff dimension $\alpha.$ In the terms used in this paper, this is equivalent to the upper and lower dimensions with respect to the family $\mathcal{F} = \set{t^\alpha: 0 < \alpha}$ coinciding. 

Measures with exact dimension 0 or 1 are often viewed as "close" to pure point or absolutely continuous measures, respectively, but they do not need to be pure point or absolutely continuous. Relevant examples include the spectral measures of 1D quasiperiodic Schr\"odinger operators. It is known that the spectral measure is 0-dimensional for every irrational frequency, yet there exist frequencies for which the measure is not pure point \cite{LanaLast2}. Much of the work applying this theory to spectral theory has been unable to address how close these measures are to these two extremes. 

It is known, however, that more general Hausdorff measures can be defined by replacing $t^\alpha$ in the definition with a suitable gauge function $\rho(t).$ 

One of the authors has explored a generalization using $\log(1/t)^{-\alpha}$ in the definition of Hausdorff measures to create a logarithmic dimension  and used it to study spectral questions in his thesis \cite{LandriganThesis}. This has already found applications in \cite{DamanikLandrigan1}.  
We take these concepts and generalize them even further using modern ideas into what we believe is the most natural general framework (complete families of Hausdorff measure functions). Using these general families of Hausdorff measures, a similar notion of dimension can be developed to address these more delicate situations. Full details are presented in section \ref{section:IntroDefs}. 

With this theory, we are able to extend the Gilbert-Pearson and Jitomirskaya-Last theories of power-law subordinacy (Theorem \ref{SUBHDIM}), and we prove that half-line operators with positive upper Lyapunov exponent have at most a logarithmic dimension (Theorem \ref{UPPERLYAP}). Moreover, we are able to construct half-line operators that achieve any given dimension for Lebesgue a.e. boundary phase (Theorem \ref{SPARSEBARTHM}). We have not extended this analysis to every boundary phase, but we believe that the removal of a null set of boundary phases is simply a limitation of our proof methods. 


Moreover, in Theorem \ref{UPPERLYAP}, we obtain a $\log\log(1/t)^2$ correction term, which is lacking in all of the existing results that just use a $\log$-dimension. 
The proof heavily relies on the fact that the spectral measure of an operator of the form \eqref{1DFORM} is supported on the set of energies, $E,$ for which there exists a solution $u_1$ to $Hu_1 = Eu_1$ satisfying:
\[u_1(0) = 0 \text{ and } u_1(1) = 1\]
and for every $\delta > 0,$
\begin{equation}\displaystyle{\limsup_{L \to \infty} \frac{\norm{u_1}^2_L}{L(\ln L)^{1 + \delta}} < \infty}, \label{eq:SCHNOL} \end{equation}
where $\norm{\cdot}_L$ is defined by \eqref{LNORM} below. This is the origin of the $\log\log(1/t)^{1 + \delta}$ correction term, and reveals when we expect this correction term to be unnecessary:  whenever we can improve \eqref{eq:SCHNOL} on certain length scales. This is precisely what we do in our proof of part 3 of Theorem \ref{SPARSEBARTHM}.
The potentials constructed in Theorem \ref{SPARSEBARTHM} do not exhibit this correction term for a.e. $\theta,$ so is would be of interest to know if there are potentials that yield spectral measures that achieve a dimension with this correction term.


The rest of our paper is organized in the following way: in section \ref{section:IntroDefs}, we build a general Hausdorff dimension framework, and introduce the major definitions, models, and results of our paper. In section \ref{section:BorelSubTheory}, we relate the Hausdorff dimension of a measure to tangential limits of its Borel transform. In section \ref{section:SubTheory}, we apply these notions to derive a subordinacy theory for half-line operators, proving Theorem \ref{SUBHDIM} and the first part of Theorem \ref{UPPERLYAP}. In section \ref{section:SparseBar}, we analyze the dimension of spectral measures associated to operators with sparse barrier potentials and prove Theorem \ref{SPARSEBARTHM}. In section \ref{section:SULE}, we discuss the behavior of spectral measures under rank-one perturbations and show that, under local perturbations, the spectrum of systems with exponentially localized eigenfunctions remains at most a logarithmic-power dimension. Finally, in section \ref{section:QuantDynam} we use our general Hausdorff dimension framework to extend the quantum dynamics theory of Last \cite{Last} to encompass our more general notion of dimension.


\section{Preliminaries and main results}\label{section:IntroDefs}

Now we will give an overview of the relevant definitions for a discussion of generalized Hausdorff dimension. 

Our analysis begins with the decomposition theory of Rogers and Taylor \cite{RT1, RT2}. The classical Lebesgue decomposition theorem provides a way to decompose any measure into three pieces: an absolutely continuous piece, a singular continuous piece, and a pure point piece. Rogers and Taylor used Hausdorff measures to further decompose the singular continuous piece.

\begin{mydef}
A Hausdorff dimension function, or gauge function, is a strictly increasing differentiable function $\rho: (0,\infty) \to (0,\infty)$ with
	$$\lim_{t\to 0^+} \rho(t) = 0.$$
\end{mydef}

\begin{mydef}
The $\rho$-dimensional Hausdorff measure, $\mu^\rho,$ is defined on the Borel $\sigma$-algebra as 
	$$\mu^\rho(F) := \lim_{\delta\to0} \inf_{\delta\text{-covers}} \set{\sum_{i = 1}^\infty \rho(|F_i|)}.$$ 
\end{mydef}

Observe that if $\rho(t) = t^\alpha$ then we arrive at the usual $\alpha$-dimensional Hausdorff measure.

Consider the family of all Hausdorff dimension functions, $\mathcal{H}$ and the partial order, $\prec,$ on $\mathcal{H}$ given by $\rho \prec \xi$ if and only if 
	\begin{equation}
	\lim_{t\to 0^+} \frac{\rho(t)}{\xi(t)} = \infty.
	\end{equation} 
It is easy to see that if $\lim_{t\to0^+} \frac{\rho(t)}{\xi(t)} = 0$ then $\xi\prec \rho.$  Additionally, we will define an equivalence relation, $\sim,$ on $\mathcal{H}$ by $\rho \sim \xi$ if and only if 
$$0 < \lim_{t\to 0^+} \frac{\rho(t)}{\xi(t)} < \infty.$$
We say that $\rho \precsim \xi$ if and only if $\rho \prec \xi$ or $\rho \sim \xi.$ 

\begin{mydef}\label{CompFam}
We say $\mathcal{F} \subset \mathcal{H}$ is a complete one-parameter family of comparable Hausdorff dimension functions if $\mathcal{F}$ is a totally ordered subset of $\mathcal{H}$ which is order isomorphic to a subinterval $I \subset \R.$ That is, if every pair $\rho, \xi\in \mathcal{F}$ obeys either $\rho\prec\xi$ or $\xi \prec\rho,$ and if there exists an interval $I\subset \R$ such that there is an order-preserving bijection from $\mathcal{F}$ to $I.$ Particularly, we can write $\mathcal{F} = \set{\rho_\alpha: \alpha \in I, \rho_\alpha \prec \rho_\beta \text{ iff } \alpha < \beta}.$ 
\end{mydef}

For simplicity, since these are the only families we will work with in this paper, we will simply call these {\it{comparable families}} or {\it{complete comparable families}.}

\begin{Remark}
We may relax the order isomorphism condition slightly to allow for order isomorphisms with boxes in $\R^n,$ for $1 \leq n \leq \aleph_0,$ along with the lexicographical order. All of our applications, however, use $n = 1.$
\end{Remark}

\begin{Remark}
We require the order isomorphism to be with an interval (or more generally, a box) to avoid the pathological behavior caused by the presence of "gaps", which can be illustrated in two examples:
\begin{enumerate}
\item First, we have the case of $\mathcal{F} = \set{t^\alpha: [0,1] \backslash \Q},$ which is unable to describe the dimension of sets with usual Hausdorff dimension $1/2,$ since the notion of supremum and infimum are not defined on $\mathcal{F}.$
\item Second, we have the case of $\mathcal{F} = \set{t^\alpha: \alpha \in (0,1/3] \cup [2/3, 1]},$ which is unable to describe the dimension of sets with usual Hausdorff dimension $1/2.$
\end{enumerate} 
\end{Remark}

From this point forwards, we will restrict our attention to such families. Typical examples include the usual functions $\set{t^\alpha}_{\alpha \in \R}$ used to define the usual Hausdorff dimension, and even more generally, families of the form $\set{\rho^\alpha}_{\alpha \in \R}$ for $\rho \in \mathcal{H}.$ 

We can use the completely ordered family $(\mathcal{F},\prec)$ to generalize the Hausdorff dimension of sets and measures in a way that reduces to the classical definition when $\mathcal{F} = \set{t^\alpha: 0 < \alpha \leq 1}:$

\begin{mydef}
Let $\mathcal{F} = \set{\rho_\alpha: \alpha \in I\subset [0,\infty), \rho_\alpha \prec \rho_\beta \text{ iff } \alpha < \beta}.$ The $\mathcal{F}$-dimension of a set $F,$ denoted $\dim_\mathcal{F}(F),$ is given by 
\begin{equation} 
\dim_\mathcal{F}(S) = \begin{cases} \rho_{\alpha'} & \alpha' \in I \\ 
0 & \alpha' = -\infty \\
1 & \text{if } \alpha' \not \in I \text{ and } \alpha' > \alpha \text{ for all } \alpha \in I
\end{cases}
\end{equation}
where $\alpha' = \sup\set{\alpha \in I: \mu^{\rho_\alpha}(S) = \infty}.$ 
\end{mydef}

Observe that this is a precise generalization of the normal Hausdorff dimension when $\mathcal{F}$ contains only functions of the form $t^\alpha.$ 

Unlike sets, a measure need not have an $\mathcal{F}$-dimension, which motivates the following definitions: 
\begin{mydef}\label{DEFCONT} \label{DEFSING}
We say that a measure $\mu$ is $\rho$-singular if there exists some set $G$ such that $\mu(\R\backslash G) = 0$ and $\mu^\rho(G) = 0.$ Similarly, we say that a measure $\mu$ is $\rho$-continuous if $\mu(S) = 0$ for every set $S$ with $\mu^\rho(S) = 0.$
\end{mydef}

This leads us the the notion of upper and lower dimension:
\begin{mydef}
The upper $\mathcal{F}$-dimension of a measure $\mu,$ denoted $\dim_\mathcal{F}^+(\mu),$ is given by
\begin{equation} 
\dim_\mathcal{F}^+(\mu) = \begin{cases} \rho_{\beta'}; & \text{if } \beta' \in I \\ 1; & \text{if } \beta' = +\infty \\ 0; & \text{if } \beta' \not \in I \text{ and } \beta' < \alpha \text{ for every } \alpha \in I \end{cases}
\end{equation}
where $\beta' = \inf\set{\alpha \in I: \mu \text{ is } \rho_\alpha\text{-singular}}.$ Similarly, we define the lower $\mathcal{F}$-dimension of a measure $\mu,$ denoted $\dim_\mathcal{F}^-(\mu)$ is given by
\begin{equation} 
\dim_\mathcal{F}^-(\mu) = \begin{cases} \rho_{\gamma'}; & \text{if } \gamma' \in I \\ 0; & \text{if } \gamma' = -\infty \\ 1 & \text{if } \gamma' \not \in I \text{ and } \gamma' > \alpha \text{ for all } \alpha \in I \end{cases}
\end{equation}
where $\gamma' = \sup\set{\alpha \in I: \mu \text{ is } \rho_\alpha\text{-continuous}}.$
\end{mydef}

We can now define the $\mathcal{F}$-dimension of a Borel measure $\mu.$ 

\begin{mydef}\label{DEFGHDIM}
The $\mathcal{F}$-dimension of a Borel measure $\mu,$ denoted $\dim_\mathcal{F}(\mu),$ is given by
\begin{equation}
\dim_\mathcal{F}(\mu) = \begin{cases} \dim_\mathcal{F}^+(\mu); & \text{ if } \dim_\mathcal{F}^+(\mu) = \dim_\mathcal{F}^-(\mu)\\
\text{undefined;} & \text{ if } \dim_\mathcal{F}^+(\mu) \ne \dim_\mathcal{F}^-(\mu)
\end{cases}.
\end{equation}
\end{mydef}

Related concepts we will occasionally use are the idea of zero-dimensional and positive-dimensional Hausdorff measure functions.

\begin{mydef}
We say a function $\rho \in \mathcal{H}$ is a zero-dimensional Hausdorff dimension function if $\rho \prec t^\alpha$ for every $\alpha > 0.$
Analogously, we say a function $\xi \in \mathcal{H}$ is a positive-dimensional Hausdorff dimension function if $t^\alpha \prec \xi$ for some $\alpha > 0.$
\end{mydef}

Our approach here is, as far as we know, novel. Past work in this direction has always dealt with studying the singularity and continuity of a measure with respect to families of the form $\set{\rho^\alpha}_{\alpha\in I},$ whereas our notion of a complete family of Hausdorff dimension functions allows us to consider more varied families, which allows us to gain sharper results.

\subsection{1D Operators}
First, we will examine dimensional properties of discrete Schr\"odinger operators on the half-line. We define
\begin{equation}\label{1DFORM}
(H_\theta \psi)(n) = \psi(n - 1) + \psi(n + 1) + V(n) \psi(n),
\end{equation}
along with a phase boundary condition
\begin{equation} \label{BOUNDARYCONDITION}
\psi(0) \cos\theta + \psi(1)\sin\theta = 0,
\end{equation}
where $-\frac \pi 2 < \theta \leq \frac \pi 2$ and the potential $V = \set{V(n)}_{n = 1}^\infty$ is a sequence of real numbers. The study of operators of the form (\ref{1DFORM}) along with the boundary condition (\ref{BOUNDARYCONDITION}) is equivalent to the study of (\ref{1DFORM}) with a Dirichlet boundary condition
\begin{align}\label{DIRICH}
\begin{split}
\psi(0) &= 0 \\
\psi(1) &= 1
\end{split}
\end{align}
along with a rank-one perturbation at the origin
\begin{equation}
V(1) \mapsto V(1) - \tan\theta.
\end{equation}
So, without loss of generality, we will confine our attention to operators of the form (\ref{1DFORM}) on $l^2(\Z^+)$ along with the Dirichlet boundary condition (\ref{DIRICH}) and interpret the boundary phase as applying the corresponding rank-one perturbation at the origin.

For these operators, it is known that the vector $\delta_1,$ which is 1 for $n = 1$ and 0 otherwise, is cyclic, so the spectral problem reduces to the study of the spectral measure $\mu = \mu_{\delta_1}.$ The behavior of this spectral measure is related to the behavior of the Weyl-Titchmarsh $m$-function, which in our case coincides with the Borel transform of $\mu:$
\begin{equation}\label{eq:BORELTRANSFORM}
F_\mu(z) = \int_\R \frac {d\mu(x)}{x - z}.
\end{equation}
When there is no ambiguity, we will usually omit the dependence on $\mu$ and express the Borel transform as $F(z).$ For a full discussion of this relationship, we refer the reader to Simon \cite{SIMON}.

Our first results will extend the Jitomirskaya-Last theory of power-law subordinacy \cite{LanaLast}, which is itself an extension of the Gilbert-Pearson theory \cite{Gilbert1, GP1, KP1}, both of which relate spectral properties of the operator (\ref{1DFORM}) to solutions of the corresponding Schr\"odinger equation
\begin{equation}\label{SCHEQN}
u(n - 1) + u(n + 1) + V(n) u(n) = Eu(n).
\end{equation}

More specifically, we will let $\norm{u}_L$ be the norm of $u$ over the lattice interval of $L.$ That is, 
\begin{equation}\label{LNORM}
\norm{u}_L = \left(\sum_{n = 1}^{\FLOOR{L}} (|u(n)|^2 + (L - \FLOOR{L})|u(\FLOOR{L} + 1)|^2)\right)^{1/2},
\end{equation}
where $\FLOOR{L}$ is the integer part of $L.$ We say that a solution $u$ of (\ref{SCHEQN}) is called subordinate if 
\begin{equation}
\lim_{n \to \infty} \frac{\norm{u}_L}{\norm{v}_L} = 0
\end{equation}
for any other linearly independent solution $v.$
The Gilbert-Pearson theory related the absolutely continuous part of the spectral measure $\mu$ to those energies $E$ for which (\ref{SCHEQN}) has no subordinate solutions; likewise, the singular part of the spectral measure $\mu$ is supported on the set of energies for which the solutions to (\ref{SCHEQN}) with the Dirichlet boundary condition are subordinate. The Jitomirskaya-Last theory refined the treatment of the singular part of the spectral measure to consider different kinds of singular-continuous spectral measures based on the classification of those measures with respect to the usual power-law Hausdorff measures and dimensions using a decomposition theory developed by Rogers and Taylor \cite{RT1, RT2}. Our treatment goes further still and considers decompositions with respect to arbitrary families of Hausdorff measures, not just the usual power-law measures.

Given $H_\theta$ of the form (\ref{1DFORM}), and $E\in \R,$ we define $u_1$ to be the solution to (\ref{SCHEQN}) obeying the Dirichlet boundary condition:
\begin{align}
\begin{split}
u_1(0) &= 0 \\
u_1(1) &= 1
\end{split}
\end{align}
and let $u_2$ be the solution of (\ref{SCHEQN}) obeying the orthogonal boundary condition:
\begin{align}
\begin{split}
u_2(0) &= 1 \\
u_2(1) &= 0.
\end{split}
\end{align}
Given any $\epsilon > 0,$ we define the length scale $L(\epsilon) \in (0,\infty)$ as the length that yields the equality 
\begin{equation}
\norm{u_1}_{L(\epsilon)}^{-1}\norm{u_2}_{L(\epsilon)}^{-1} = 2\epsilon.
\end{equation}

Another useful tool in studying operators of the form (\ref{1DFORM}) is the $n$-step {\it transfer matrix } $\Phi_n(\theta, E).$ This is the matrix 
\begin{equation} \Phi_n(\theta, E) = \begin{pmatrix} u_1(n+1) & u_2(n+1) \\ u_1(n) & u_2(n)\end{pmatrix}.\end{equation}
With this, we can define the {\it upper Lyapunov exponent}, 
\begin{equation}\label{Lupper}
L^*(\theta, E) = \limsup_{n \to \infty} \frac 1 n \ln \norm{\Phi_n(\theta, E)}.
\end{equation}

We know of no explicit link between the operator $H$ and the local scaling behavior of the spectral measure $\mu$ in this regime, so we begin with an important technical result relating the generalized Hausdorff dimension of a Borel measure to growth properties of its Borel transform:
\begin{mythm}
\label{HAUSBOREL}
Define $A_0 = \set{0}, A_1 = (0,\infty), $ and $A_2 = \set{\infty}.$  Suppose $\rho$ is a Hausdorff dimension function satisfying $\rho(t) \prec t.$ We have 
	$$\limsup_{\epsilon \to 0^+} \frac{\mu((x - \epsilon,x+\epsilon))}{\rho(\epsilon)} \in A_i$$ 
if and only if $$\limsup_{\epsilon \to 0^+} \frac{\epsilon}{\rho(\epsilon)}\IM F(x + i\epsilon) \in A_i,$$ with $i = 0,1,2.$
\end{mythm}

Our first core result is a subordinacy theory extending the work of Jitomirskaya-Last \cite{LanaLast, LanaLast2}, which links the generalized Hausdorff dimension of the spectral measure $\mu$ to growth properties of $u_1$ and $u_2:$
\begin{mythm}\label{SUBHDIM}
Let $u_1$ and $u_2$ be solutions of the equation $Hu = Eu$ for $E \in \R$ obeying $u_1(0) = 0, u_1(1) = 1, u_2(0) = 1, $ and $u_2(1) = 0.$ Let $\rho(t)$ be a Hausdorff measure function. We have 
	$$\limsup_{\epsilon \to 0} \frac{\epsilon}{\rho(\epsilon)}F(E + i\epsilon) = \infty$$ 
if and only if 
	$$\liminf_{L\to \infty} \rho(\norm{u_1}^{-1}_L\norm{u_2}^{-1}_L)\norm{u_1}_L^2 = 0.$$
\end{mythm}

Our second key result is a bound on the upper spectral dimension of a half-line operator with positive upper Lyapunov exponent:
\begin{mythm} \label{UPPERLYAP}
Let $\mathcal{F} = \set{\rho_\alpha: \alpha \in I}$ be a family of comparable Hausdorff dimension functions such that for some $\delta, \epsilon > 0,$ we have $f_\delta(t), g(t)^{1 - \epsilon}  \in \mathcal{F},$ where 
\begin{equation}
f_\delta(t) = \frac{1}{\ln(1/t)(\ln(\ln(1/t)))^{1 + \delta}}
\end{equation}
and
\begin{equation}
g(t) = \frac{1}{\ln(1/t)}.
\end{equation}
 If the upper Lyapunov exponent is positive for every $E$ in some Borel set $A$ then $\dim_\mathcal{F}^+(\mu(A\cap \cdot)) \precsim f_\delta(t).$ Moreover, there exists an operator of the form (\ref{1DFORM}) with positive upper Lyapunov exponent whose spectral measure $\mu$ obeys $g(t)^{1 - \epsilon} \precsim \dim_\mathcal{F}^-(\mu).$
\end{mythm}

This improves upon an earlier result from \cite{LanaLast} which was only able to conclude that the power-law dimension was 0, and earlier results from Landrigan, which were only able to conclude that the $\log$-dimension was at most 1. There are two immediate consequences of this result: (i) the lower dimension bound shows that Landrigan's $\log$-dimension result is sharp, (ii) our upper dimension result reveals that there may be examples of operators with positive upper Lyapunov exponent that have a dimension strictly larger than $1/\ln(1/t).$ 

While we do not know which of these bounds is sharp, by considering an operator with a suitably sparse barrier potential, we are able to show that the lower dimension above can coincide with the actual dimension for Lebesgue a.e. boundary phase. Let $\beta(x)$ be a non-negative increasing convex function such that $\log(\beta(x))$ is still convex. For example, we could take $\beta(x) = e^x.$ Moreover, suppose that $G(t) = 1/\beta^{-1}(1/t^2)$ defines a zero dimensional Hausdorff dimension function. We will consider any family of Hausdorff dimension functions, $\mathcal{F},$ such that $G(t)^{1/\eta},$ $G(t)^{(1 - \epsilon)/\eta}$ and $G(t)^{1/\eta}/(\ln(\beta^{-1}(1/t)))^{1 + \delta} \in \mathcal{F}$ for some $\eta > 0,$ and some $\epsilon > 0.$ Define length scales inductively by $L_1 = 2, L_{n+1} = \beta(L_n)^n$ and define a potential
	\begin{equation}
	V(n) = \begin{cases} \beta(L_k)^\eta & n = L_k \\
	0 & n\not\in \set{L_k}_{k = 1}^\infty
	\end{cases}.
	\end{equation}
	
\begin{mythm}\label{SPARSEBARTHM}
Let $\eta, \beta, G(t), \mathcal{F}$ and $V(n)$ be as above. Let $\mu_\theta$ be the spectral measure of the half-line operator $(H_\theta u)(n) = u(n+1) + u(n-1) + V(n)u(n)$ with boundary phase $\theta.$ Then 
\begin{enumerate}
\item[(i)] for every boundary phase $\theta,$ the spectrum of $H_\theta$ consists of the interval $[-2,2]$ (which is the essential spectrum) along with some discrete point spectrum outside this interval;
\item[(ii)] for every $\theta,$  
\begin{align*}
G(t)^{(1 - \epsilon)/\eta} &\precsim \dim_\mathcal{F}^-(\mu_\theta((-2,2)\cap\cdot))\\
&\precsim \dim_\mathcal{F}^+(\mu_\theta((-2,2)\cap\cdot)) \\
&\precsim G(t)^{1/\eta}/(\ln(\beta^{-1}(1/t)))^{1 + \delta};
\end{align*}
\item[(iii)] for Lebesgue a.e. $\theta,$ $\dim^+_\mathcal{F}(\mu) \precsim G(t)^{1/\eta}.$
\end{enumerate}
\end{mythm}

In particular, if we take $\beta(x) = e^x, \eta = 1,$ then Theorem \ref{SPARSEBARTHM} proves the second part of Theorem \ref{UPPERLYAP}.

One of the most interesting parts of this theorem is that we only prove an exact dimension result for a.e. boundary phase $\theta.$ This is a limitation of our proof, where we carefully study the existence of suitably decaying solutions in the case $\theta = 0,$ and interpret different boundary phases as consequences of particular rank-one perturbations; by considering rank-one perturbations, we are able to deduce the existence of similarly decaying solutions for other boundary phases, but lose a Lebesgue null set in the process. A similar result is known to hold for every boundary phase when positive power-law Hausdorff dimensions are considered, but the only proof we are aware of requires more involved arguments involving quantum dynamics \cite{TCHPotentials}.

\subsection{Systems with exponentially localized eigenfunctions}
We then turn our attention to fractal properties of Schr\"odinger operators on the lattice $l^2(\Z^\nu), \nu \geq 1.$

First, we study what happens to the dimensional properties of spectral measures when we apply rank-one perturbations to operators with exponentially localized eigenfunctions. More specifically, by the spectral theorem it is known that every bounded self-adjoint operator on a Hilbert can be realized as $A: L^2(d\mu) \to L^2(d\mu),$ $\psi \mapsto \psi \cdot x,$ for some suitable measure $\mu.$  If we let $\varphi \in L^2(d\mu)$ be a cyclic unit vector, then we can easily define the rank-one perturbation of $A$ by $\varphi$ as
\begin{equation}
A_\lambda = A + \lambda \INNERPROD{\varphi}{\cdot}\varphi, \quad \lambda \in \R.
\end{equation}
If we let $\mu_\lambda$ denote the spectral measure of $A_\lambda$ associated to $\varphi,$ and $F_\lambda$ the Borel transform of $\mu_\lambda,$ then it is known that
\begin{equation}
F_\lambda(z) = \frac{F_0(z)}{1+ \lambda F_0(z)}
\end{equation}
which, in conjunction with our work relating dimensional properties of a measure to growth properties of Borel transforms, allows us to study how the dimension of a spectral measure is affected when it is under the effect of a rank-one perturbation. 

We say that a self-adjoint operator on $l^2(\Z^\nu)$ has semi-uniformly localized eigenfunctions (SULE) if and only if  there is a complete set of orthonormal eigenfunctions, $\set{\varphi_n}_{n = 1}^\infty,$ there is $\alpha > 0$ and $m_n \in \Z^\nu, n \geq 1$ and for each $\delta > 0,$ a $C_\delta > 0$ so that 
\begin{equation}
|\varphi_n(m)| \leq C_\delta e^{\delta |m_n| - \alpha|m - m_n|}
\end{equation}
for all $m \in \Z^\nu,$ and $n \geq 1.$ 

It is known, \cite{DelRio1}, that if an operator $H: l^2(\Z^\nu) \to l^2(\Z^\nu)$ has SULE, if $H_\lambda = H + \lambda\INNERPROD{\delta_0}{\cdot}\delta_0,$ and if $\mu$ and $\mu_\lambda$ are the spectral measure for $H$ and $H_\lambda$ respectively associated to $\delta_0,$ then $\mu_\lambda$ is zero-dimensional. We are able to improve this into
\begin{mythm} \label{MySULERes}
Suppose $H$ has SULE and let $\mathcal{F} = \set{\ln(1/t)^{-\alpha}: 0 < \alpha < \infty}.$ Let $H_\lambda = H + \lambda\INNERPROD{\delta_0}{\cdot}\delta_0.$ Let $d\mu$ be the spectral measure of $H$ associated to $\delta_0,$ and let $d\mu_\lambda$ be the corresponding spectral measures for $H_\lambda.$ Then for every $\lambda,$ $\dim_\mathcal{F}(\supp(d\mu_\lambda)) \precsim \ln(1/t)^{-\nu}.$ 
\end{mythm}


\subsection{Quantum dynamics}
We now turn our attention to dynamical properties of Schr\"odinger operators on the lattice $l^2(\Z^\nu), \nu \geq 1.$ Our main interest in this setting is in dynamical properties of operators of the form
\begin{equation}
(H \psi)(n) = \sum_{|n - m| = 1} \psi(m) + V(n) \psi(n),
\end{equation} 
though much of our discussion applies to any self-adjoint Hamiltonian. A theory based on the power-law dimension was developed by Last \cite{Last}. Notably, the theory establishes an extremely useful connection between the continuity of a spectral measure and the average growth of the moments of the corresponding position operator (see, e.g. \cite{DamTch1}, \cite{KKLDynamics}, \cite{Tch2}, \cite{TCHPotentials}). Our starting point is that the original theory of Rogers and Taylor was actually developed in the generality that we are using; in particular, the decomposition theory and the critical Theorem 4.2 in \cite{Last} exists in our general setting once a suitable notion of uniform H\"older continuity with respect to a general Hausdorff dimension function is realized. This allows us to proceed in much the same manner as Last. 

A common application of Last's theory is the notion of a transport exponent, which relates to the average power-law growth of the $p^{th}$ moment of the position operator. One of our most important results in this direction is
\begin{mythm} \label{POSBDTHM} If $H$ is self-adjoint on $l^2(\Z^\nu)$ and $P_{\rho c}\psi \ne 0,$ where $P_{\rho c}$ is the orthogonal projection on $\mathscr{H}_{\rho c},$ then for each $m > 0,$ there exists a constant $C = C(\psi,m)$ such that for every $T > 1$
\begin{equation}
\left\langle\left\langle |X|^m \right\rangle\right\rangle_T > C \rho(1/T)^{-m/\nu}.
\end{equation}
\end{mythm}
This may be used to define a more general notion of transport exponent than has previously been studied. Analysis of this transport exponent has been of central importance in many dynamical results (see e.g. \cite{DamTch1}, \cite{KKLDynamics}, \cite{Tch2}, \cite{TCHPotentials}) and we hope to extend this analysis in future work.

\section{General Hausdorff dimension of sets and measures}\label{section:BorelSubTheory}

The following characterization dates back to the original work of Rogers and Taylor \cite{RT1, RT2}
\begin{mythm}\label{RTDef}
Let $A$ be a Borel set, and let $\mu$ be a Borel measure. Then \begin{enumerate}
	\item $\mu(\cdot \cap A)$ is $\rho$-singular if and only if 
		$$\limsup_{\epsilon \to 0}\frac{\mu(x - \epsilon, x + \epsilon)}{\rho(\epsilon)} = \infty$$ 
	for $\mu$-a.e. $x\in A.$
	\item $\mu(\cdot \cap A)$ is $\xi$-continuous if and only if 
		$$\limsup_{\epsilon \to 0}\frac{\mu(x - \epsilon, x + \epsilon)}{\xi(\epsilon)} < \infty$$ 
	for $\mu$-a.e. $x\in A.$
\end{enumerate}
\end{mythm}

When $\mu$ is the spectral measure of some self-adjoint operator $A,$ we know of no direct relation between the local scaling behavior of $\mu$ and spectral properties of $A.$ To bridge the gap between the the two, we will need to introduce the Borel transform, as in \cite{DelRioJitLastSim}:

\begin{mydef}
The Borel transform of a measure $\mu,$ denoted $F_\mu(z),$ is $$F_\mu(z) = \int_\R \frac{d\mu(x)}{x-z}.$$
\end{mydef}

It is known that the Borel transform provides an alternate characterization to Theorem \ref{RTDef} for the usual Hausdorff dimension, but it in fact applies to our more general notion. Notably, we may now prove Theorem \ref{HAUSBOREL}:

\begin{proof}[Proof of Theorem \ref{HAUSBOREL}.]


	Let $M_\mu^\delta(x_0) = \mu(x_0 - \delta, x_0 + \delta).$ By definition, we have 
	\begin{equation}
	\IM F_\mu(x_0 + i\epsilon) = \epsilon \int_{-\infty}^\infty \frac{d\mu(y)}{(y - x_0)^2 + \epsilon^2} \geq \frac 1 {2\epsilon} M_\mu^\epsilon(x_0),
	\end{equation}
	so 
	\begin{equation}
	\frac{M_\mu^\epsilon(x_0)}{\rho(\epsilon)} \leq 2\frac{\epsilon}{\rho(\epsilon)} \IM F_\mu(x_0 + i\epsilon).
	\end{equation}
	Thus 
	\begin{equation}
	\limsup_{\epsilon \to 0^+} \frac{M_\mu^\epsilon(x_0)}{\rho(\epsilon)} \leq 2\limsup_{\epsilon \to 0^+} \frac{\epsilon}{\rho(\epsilon)} \IM F_\mu(x_0 + i\epsilon).
	\end{equation}
	Hence, if the LHS = $\infty,$ then so does the RHS. Analogously, if the RHS = 0, then so does the LHS. 
	
	On the other hand, suppose $\limsup_{\epsilon \to 0} \frac{M_\mu^\epsilon(x_0)}{\rho(\epsilon)} < \infty.$ Then we know that 
	\begin{equation}
	M_\mu^\delta(x_0) \leq C \rho(\delta),
	\end{equation}
	for $\delta$ sufficiently small, so we have
	\begin{align}\label{COMPFORCOR}
	\limsup_{\epsilon \to 0^+} \frac{\epsilon}{\rho(\epsilon)} \IM F_\mu(x_0 + i \epsilon) &\leq \limsup \frac{\epsilon}{\rho(\epsilon)} |F_\mu(x_0 + i\epsilon)| \\
	&\leq \limsup \frac{\epsilon}{\rho(\epsilon)} \int_{-\infty}^\infty \frac{d\mu(y)}{[(x_0 - y)^2 + \epsilon^2]^{1/2}} \\
	&= \limsup \frac{\epsilon}{\rho(\epsilon)} \left(\int_{|y - x_0| > 1} + \int_{|y - x_0| \leq 1} \right) \\
	&= \limsup \frac{\epsilon}{\rho(\epsilon)} \int_{|y - x_0| \leq 1} \frac{d\mu(y)}{[(x_0 - y)^2 + \epsilon^2]^{1/2}}\label{eq:LASTLINE1}.
	\end{align}
Here \eqref{eq:LASTLINE1} follows from the observation that 
\begin{equation}
	\lim \frac{\epsilon}{\rho(\epsilon)} \int_{|y - x_0| > 1} \frac{d\mu(y)}{[(x_0 - y)^2 + \epsilon^2]^{1/2}} = 0.
\end{equation}
	
We can then evaluate the remaining integral by integrating by parts, and by observing that the boundary term at 0 vanishes:
	\begin{align}
	\limsup \frac{\epsilon}{\rho(\epsilon)} \int_{|y - x_0| \leq 1} \frac{d\mu(y)}{[(x_0 - y)^2 + \epsilon^2]^{1/2}} &= \limsup \frac{\epsilon}{\rho(\epsilon)} \int_0^1 \frac{\delta}{(\epsilon^2 + \delta^2)^{3/2}}M_\mu^\delta(x_0) d\delta \\
	&\leq \limsup C \frac{\epsilon}{\rho(\epsilon)} \int_0^1 \frac{\delta \rho(\delta)}{(\epsilon^2 + \delta^2)^{3/2}} d\delta.
	\end{align}
	Now we break the integral into two pieces:  $\int_0^\epsilon + \int_\epsilon^1$ and observe that the first piece is uniformly bounded. The second piece can be bounded as
	\begin{equation}
	\limsup C \frac{\epsilon}{\rho(\epsilon)} \int_\epsilon^1 \frac{\delta \rho(\delta)}{(\epsilon^2 + \delta^2)^{3/2}} d\delta \leq \frac{C}{\rho(\epsilon)}\int_\epsilon^1 \frac{\rho(\delta)}{\delta}d\delta < \infty. \label{ENDCOMPFORCOR}
	\end{equation}
	This is finite after an application of L'Hospital's rule and the assumption that $\rho(t)\prec t.$
	
	
	This finally implies that the two desired $\limsup$ are either both finite or infinite. Moreover, if $\limsup_{\epsilon \to 0^+} \frac{M_\mu^\epsilon(x_0)}{\rho(\epsilon)} = 0$ then the constant $C$ above may be taken to be arbitrarily small, ensuring that $\limsup_{\epsilon \to 0^+} \frac{\epsilon}{\rho(\epsilon)} \IM F_\mu(x_0 + i \epsilon)  = 0.$ This completes our proof.
	
	\end{proof}



\section{Half-line subordinacy} \label{section:SubTheory}

Let $H_\theta$ be the self-adjoint operator defined on $l^2(\Z^+)$ by 
	\begin{equation}
	(H_\theta u)(n) = u(n - 1) + u(n+1) + V(n)u(n),
	\end{equation}
where $\set{V(n)}_{n = 1}^\infty$ is a sequence of real numbers along with the phase boundary condition 
	\begin{equation}
	u(0)\cos\theta + u(1)\sin\theta = 0.
	\end{equation}

\begin{mydef} We define the length scale $L(\epsilon)$ as the length that yields the equality $\norm{u_1}_{L(\epsilon)}^{-1}\norm{u_2}_{L(\epsilon)}^{-1} = 2\epsilon.$
\end{mydef}

\begin{mythm}\label{SUBHDIMVar}
Let $u_1$ and $u_2$ be solutions of the equation $Hu = Eu$ for $E \in \R$ obeying $u_1(0) = 0, u_1(1) = 1, u_2(0) = 1, $ and $u_2(1) = 0.$ Let $\rho(t)$ be a Hausdorff measure function. We have 
	$$\limsup_{\epsilon \to 0} \frac{\epsilon}{\rho(\epsilon)}F(E + i\epsilon) = \infty$$ 
if and only if 
	$$\liminf_{L\to \infty} \rho(\norm{u_1}^{-1}_L\norm{u_2}^{-1}_L)\norm{u_1}_L^2 = 0.$$
\end{mythm}

\begin{proof}
This follows from Theorem 1 of \cite{LanaLast} and Theorem \ref{HAUSBOREL} above. 
\end{proof}

We now have two applications of this theorem to zero dimensional Hausdorff dimension functions and positive dimension Hausdorff dimension functions.

\begin{mythm}
\label{GENSING}
Let $f(L)$ be a continuous, strictly increasing function such that (1) $f(0) \geq 0$ (2) $\lim_{L\to\infty}f(L) = \infty$ and (3) $\lim_{L\to \infty}\frac{t^\alpha}{f(t)} = 0$ for every $\alpha \geq 1,$ and let $g(t) = \frac{1}{t(\ln t)^{1 + \delta}}.$ Suppose that for every $E$ in some Borel set $A,$ we can find a solution, $v = au_1 + bu_2,$ to $Hv = Ev$ that satisfies 
	$$\limsup_{L\to\infty} \frac{\norm{v}^2_{L}}{f(L)} \geq 1.$$ 
Let $\mathcal{F}$ be any family of comparable Hausdorff dimension functions that contains $g(f^{-1}(\frac{|b|^2}{1 - |b|^2\epsilon}t^{-2})),$ for some constant $|b|$ and $0 < \epsilon < 1/|b|^2.$ Then $\dim_\mathcal{F}^+(\mu(A\cap\cdot)) \precsim g(f^{-1}(\frac{|b|^2}{1 - |b|^2\epsilon}t^{-2})).$
\end{mythm}

\begin{proof}
	It is known that $\mu$ is supported on the set of energies $E$ for which $u_1$ satisfies the inequality 
		\begin{equation}\label{U1Condition}\limsup_{L\to\infty} \frac{\norm{u_1}^2_L}{L(\ln L)^{1 + \delta}} < \infty,\end{equation} 
	for every $\delta > 0,$ so we may restrict our attention to those energies. For every $E\in A, v$ must be a linear combination of $u_1$ and $u_2,$ say $v = au_1 + bu_2.$ Thus, for every $L,$ 
		$$\norm{v}_L \leq |a|\norm{u_1}_L + |b|\norm{u_2}_L.$$
	By our choice of $f,$ and our restriction on the energies, $E,$ we see that we must have $b\ne 0,$ so 
		$$\norm{u_2}_L \geq \frac{\norm{v}_L - |a|\norm{u_1}_L}{|b|}.$$
	Hence, we must also have 
		\begin{equation}\label{U2Condition}\limsup_{L\to \infty} \frac{\norm{u_2}_L^2}{f(L)} \geq \frac 1 {|b|^2}\end{equation}
	for all such $E.$ 
	Now (\ref{U1Condition}) implies 
		\begin{equation}\norm{u_1}_L < CL^{1/2}(\ln L)^{(1 + \delta)/2}\end{equation}
	for some constant $C> 0,$ and (\ref{U2Condition}) implies
		\begin{equation}\norm{u_2}_{L_n}^2 > \left(\frac{1}{|b|^2} - \epsilon\right) f(L_n)\end{equation}
	for some sequence $L_n \to \infty$ and every $0 < \epsilon < \frac 1 {|b|^2}.$ 

	Now consider $\epsilon_n$ such that $L_n = L(\epsilon_n).$ We have
		\begin{align}
		g(f^{-1}(&\frac{|b|^2}{1 - |b|^2\epsilon}\norm{u_1}^2_{L_n}\norm{u_2}_{L_n}^2))\norm{u_1}_{L_n}^2 \\
		&= \frac {\norm{u_1}^2_{L_n}}{f^{-1}(\frac{|b|^2}{1 - |b|^2\epsilon}\norm{u_1}^2_{L_n}\norm{u_2}_{L_n}^2)(\ln f^{-1}(\frac{|b|^2}{1 - |b|^2\epsilon}\norm{u_1}^2_{L_n}\norm{u_2}_{L_n}^2))^{1 + \delta}}\\
		&\lesssim \frac{L_n(\ln{L_n})^{1 + \delta}}{f^{-1}(f(L_n))\ln(f^{-1}(f(L_n)))^{1+\delta}}\\
		&= 1.
		\end{align}
	Thus, if $g(f^{-1}(\frac{|b|^2}{1 - |b|^2\epsilon}t^{-2})) \prec \rho(t)$ it is easy to see that 
	$$\lim_{L_n\to \infty} \rho(\norm{u_1}^2_{L_n}\norm{u_2}_{L_n}^2)\norm{u_1}_{L_n}^2 = 0.$$
	We now finish by appealing to Theorem \ref{SUBHDIM}.
	\end{proof}



\begin{mycor}\label{TRANSFERMATRIX}
Let $f(t), g(t),$ and $\mathcal{F}$ be as in Theorem \ref{GENSING}. Let $\Phi_n(\theta, E)$ be the $n$-step transfer matrix associated to $Hu = Eu$ along with the boundary condition $\theta.$ Suppose $$\limsup_{L\to\infty} \frac{1}{f(L)} \sum_{n = 1}^L \norm{\Phi_n(\theta,E)}^2 \geq 2$$ for every $E$ in some Borel set $A.$ Then $\dim_\mathcal{F}^+(\mu(A \cap \cdot)) \precsim g(f^{-1}(\frac{1}{1 - \epsilon}t^{-2})).$
\end{mycor}

\begin{proof}
	Recall that 
		\begin{equation} \Phi_n(\theta,E) = \begin{pmatrix} u_1(n+1) & u_2(n+1) \\ u_1(n) & u_2(n)\end{pmatrix}\end{equation}
	so 
		\begin{equation} \norm{\Phi_n(\theta,E)}^2 \leq |u_1(n+1)|^2 + |u_1(n)|^2 + |u_2(n+1)|^2 + |u_2(n)|^2,\end{equation}
	and summing yields
		\begin{equation} \sum_{n = 1}^L \norm{\Phi_n(\theta,E)}^2 \leq 2(\norm{u_1}_L^2 + \norm{u_2}_L^2).\end{equation}
	Thus we conclude that (\ref{U2Condition}) holds, so Theorem \ref{GENSING} yields our result.
\end{proof}

Using Corollary \ref{TRANSFERMATRIX}, we can now prove Theorem \ref{UPPERLYAP}:

\begin{proof}[Proof of Theorem \ref{UPPERLYAP}.]
	Positive upper Lyapunov exponent yields 
		$$\norm{\Phi_{n_k}(\theta,E)} > e^{n_k(L^*(E)/2)}$$
	for some subsequence $n_k.$ Thus we may apply the corollary with $f(L) = e^{L(L^*(E)/2)}$ to see that 
	$$\dim_\mathcal{F}^+(\mu(A\cap\cdot)) \prec \frac{2}{L^*(E)\ln(\frac{1}{1-\epsilon}t^{-2})(\ln(2\ln(\frac{1}{1 - \epsilon}t^{-2}))^{1 + \delta}/L^*(E))}.$$
	Since 
	$$\frac{2}{L^*(E)\ln(\frac{1}{1-\epsilon}t^{-2})(\ln(2\ln(\frac{1}{1 - \epsilon}t^{-2}))^{1+\delta}/L^*(E))} \sim \frac{1}{\ln(1/t)(\ln(\ln(1/t)))^{1 + \delta}},$$
	we are done.
	The second part of the theorem follows from Theorem \ref{SPARSEBARTHM}, which is proved in the next section.
\end{proof}


We also have 
\begin{mythm}
\label{POSDIM}
Let $f(L) = L^{g(L)}$ be a continuous, strictly increasing function such that (1) $f(0) \geq 0$ and (2) $\limsup_{L\to \infty}g(t) = \alpha \in (1,\infty).$ Let 
\begin{equation}
\rho_\beta(t) = \frac{t^{2/(1+\beta)}}{\ln^{2\beta/(1 + \beta)} t}
\end{equation}
and suppose $\mathcal{F}$ is a family of comparable Hausdorff dimension functions that contains $\rho_\beta$ for some $\beta < \alpha.$ Suppose that for every $E$ in some Borel set $A,$ we can find a solution, $v,$ to $Hv = Ev$ that satisfies 
	$$\limsup_{L\to\infty} \frac{\norm{v}^2_{L}}{f(L)} > 0.$$ 
Then $\dim_\mathcal{F}^+(\mu(A\cap\cdot)) \precsim \rho_\beta(t).$ 
\end{mythm}

\begin{proof}
As before, we have the following bounds on $\norm{u_1}$ and $\norm{u_2}:$
\begin{equation}
	\norm{u_1}_L \lesssim L^{1/2}\ln L
\end{equation}
and
\begin{equation}
	\norm{u_2}_{L_n}^2 \gtrsim f(L_n)
\end{equation}
for some sequence $L_n \to \infty.$
Taken together, we have
	\begin{align}
	\rho_\beta(\norm{u_1}_{L_n}^{-1}\norm{u_2}_{L_n}^{-1})\norm{u_1}_{L_n}^2 &= \frac{\norm{u_1}_{L_n}^{2\beta/(1 + \beta)}}{\norm{u_2}_{L_n}^{2/(1 + \beta)}\ln^{2\beta/(1+\beta)}(\norm{u_1}_{L_n}\norm{u_2}_{L_n})2/(1+\beta)} \\
	&\lesssim \frac{L_n^{\beta/(1 + \beta)} \ln(L_n)^{2\beta/(1 + \beta)}}{L_n^{g(L_n)/(1 + \beta)} (g(L_n)\ln(L_n))^{2\beta/(1 + \beta)}}\\
	&= \frac{L_n^{\frac{\beta - g(L_n)}{1 + \beta}}}{g(L_n)^{\frac{2\beta}{1 + \beta}}} \to 0.
	\end{align}
	We now finish by appealing to Theorem \ref{SUBHDIM}.
	\end{proof}
 
 There is also an analogous version of Corollary \ref{TRANSFERMATRIX}.

\section{Proof of Theorem \ref{SPARSEBARTHM}} \label{section:SparseBar}

	
Let $\beta(x)$ be a non-negative increasing convex function such that $\log(\beta(x))$ is still convex. For example, we could take $\beta(x) = e^x.$ Moreover, suppose that $G(t) = 1/\beta^{-1}(1/t^2)$ defines a zero dimensional Hausdorff dimension function. We will consider the family of Hausdorff dimension functions $\mathcal{F} = \set{G(t)^\alpha: 0 < \alpha < \infty}.$ Define length scales inductively by $L_1 = 2, L_{n+1} = \beta(L_n)^n$ and define a potential
	\begin{equation}
	V(n) = \begin{cases} \beta(L_k)^\eta & n = L_k \\
	0 & n\not\in \set{L_k}_{k = 1}^\infty
	\end{cases}.
	\end{equation}
	
We will begin with an elementary lemma which will be useful:
\begin{mylemma}\label{CONVEXITY}
Let $\beta(x)$ be defined as above. Then $\beta^{-1}(x y) \leq \beta^{-1}(x) + \beta^{-1}(y)$ for every $x,y \geq 0.$ 
\end{mylemma}

\begin{proof}
	Since $\beta$ is increasing, this inequality is trivial if $x = 0$ or $y = 0,$ so suppose $x,y > 0.$ Then we can write $x = e^{x_1}$ and $y = e^{y_1}$ for some $x_1 \ne -y_1 \in \R.$ Since $\log(\beta(x))$ is convex, the inverse, $\beta^{-1}(e^x),$ is concave. Define $f(x) = \beta^{-1}(e^x).$ Then we have:
	\begin{equation}
	\beta^{-1}(xy) = f(x_1 + y_1).
	\end{equation}
Since positive concave functions are subadditive, we have
	\begin{align}
	f(x_1 + x_2) &\leq f(x_1) + f(x_2)\\
	&= \beta^{-1}(x) + \beta^{-1}(y).
	\end{align}
Hence $\beta^{-1}(xy) \leq \beta^{-1}(x) + \beta^{-1}(y).$
\end{proof}

Now we may turn our attention to a proof of Theorem \ref{SPARSEBARTHM}:

\begin{proof}[Proof of Theorem \ref{SPARSEBARTHM}(i).] 
	It is well-known \cite{LanaLast} that the essential spectrum is contained in the interval $[-2,2],$ so it just remains to prove the dimension result. 
\end{proof}

\begin{proof}[Proof of Theorem \ref{SPARSEBARTHM}(ii).]
	For simplicity, we will prove the theorem for $\eta = 1.$ The proof of the general case is similar.
	Let $I = [a,b] \subset (-2,2)$ It suffices to prove the theorem for $\mu_\theta(I\cap\cdot).$  
	
	Let $\alpha(x)$ be defined such that $\beta(x) = x^{\alpha(x)}.$ 

	First, we will prove that $\dim_\mathcal{F}^+(\mu_\theta(I\cap\cdot)) \precsim \beta^{-1}(1/t^2)^{-1/(1-\epsilon)}$ for every $\epsilon > 0$ and every boundary phase $\theta.$ For every $E\in I, m > k \geq 0,$ let
		\begin{equation}
		\Phi_{k,m}(E) = T_m(E)T_{m-1}(E)\cdots T_{k+1}(E)
		\end{equation}
	where 
		\begin{equation}
		T_n(E) = \begin{pmatrix} E - V(n) & -1 \\ 1 & 0\end{pmatrix}.
		\end{equation}
	Since $\det(\Phi_{k,m}(E)) = 1,$ it follows that $\norm{\Phi_{k,m}^{-1}(E)} = \norm{\Phi_{k,m}(E)}^{-1}.$ For any $n\in \Z^+,$ if $L_n \leq k < m < L_{n+1},$ then $\Phi_{k,m}(E)$ is the same as the transfer matrix for the free Laplacian. In particular, there exists some constant $C_I,$ depending only on the interval $I,$ such that $1 \leq \norm{\Phi_{k,m}(E)} \leq C_I$ for any such $k,m$ and $E\in I.$ Moreover, for any $n\in \Z^+,$ we have
		\begin{equation}
		\Phi_{L_n-1,L_n}(E) = T_{L_n}(E) = \begin{pmatrix} E - V(L_n) & -1 \\ 1 & 0 \end{pmatrix},
		\end{equation}
	and so
		\begin{equation}
		\max\set{1,V(L_n)-2} \leq \norm{T_{L_n}(E)} \leq V(L_n) + 3.
		\end{equation}
	If we consider some $n\in \Z^+$ and $L_n \leq m < L_{n+1},$ then we have
		\begin{equation}
		\Phi_{0,m}(E) = \Phi_{L_n,m}T_{L_n}\Phi_{L_{n-1},L_n-1} T_{L_{n-1}} \cdots \Phi_{L_1,L_2-1}T_{L_1}\Phi_{0,L_1-1}.
		\end{equation}
	Thus we see that 
		\begin{align}
		\begin{split}
		\norm{\Phi_m(E)} &\leq C_I^{n+1}\prod_{k = 1}^n(V(L_k)+3) \\
		&\leq C_1^n\prod_{k = 1}^n L_k^{\alpha(L_k)} \\
		&\leq C_1^n \left(\prod_{k = 1}^n L_k\right)^{\alpha(L_n)},
		\end{split}
		\end{align}
	where $C_1$ is some constant. Similarly, for large $n,$ we also have
		\begin{align}
		\begin{split}
		\norm{\Phi_m(E)} &\geq \left(C_I^{n+1}\prod_{k = 1}^{n-1}(V(L_k)+3)\right)^{-1}(V(L_n)-2) \\
		&\geq C_2^{-n}\left(\left(\prod_{k = 1}^{n-1}L_k\right)^{-1}L_n\right)^{\alpha(L_n)},
		\end{split}
		\end{align}
	where $C_2$ is some constant. Since $L_{n+1} = \beta(L_n)^n,$ we see that for any $\epsilon > 0$ we can take $n$ sufficiently large so that 
		\begin{equation}
		\frac{L_n}{\beta^{-1}(L_n)} < \left(\prod_{k=1}^{n-1} L_k\right)^{-1}L_n < \prod_{k = 1}^n L_k < L_n\beta^{-1}(L_n).
		\end{equation}
	Similarly, for any $\epsilon > 0$ and $n$ large enough we have $C_1^n < \beta^{-1}(L_n)^\epsilon$ and $C_2^n < \beta^{-1}(L_n)^\epsilon.$ Hence, for any $L_n \leq m < L_{n + 1},$
		\begin{equation}\label{EXPBOUND}
		\left(\frac{L_n}{\beta^{-1}(L_n)}\right)^{\alpha(L_n)}\beta^{-1}(L_n)^{-\epsilon} \leq \norm{\Phi_m(E)}\leq \left(L_n\beta^{-1}(L_n)\right)^{\alpha(L_n)} \beta^{-1}(L_n)^{\epsilon}.
		\end{equation}
	Set $h(m) = \beta(m).$ By taking $m = L_n,$ we have
		\begin{align}\label{eq:tranmatbd}
		\begin{split}
		\frac{1}{h(m)}\sum_{k = 1}^m \norm{\Phi_k}^2 &\geq L_n^{-\alpha(L_n)}\left(\frac{L_n}{\beta^{-1}(L_n)}\right)^{2\alpha(L_n)}\beta^{-1}(L_n)^{-2\epsilon}\\
		&= L_n^{\alpha(L_n)}\beta^{-1}(L_n)^{-2\alpha(L_n) - 2\epsilon} \\
		\end{split}
		\end{align}
		By our assumptions on $\beta,$ \eqref{eq:tranmatbd} $\to \infty.$
	Corollary \ref{TRANSFERMATRIX} now yields $$\dim_\mathcal{F}^+(\mu_\theta(I\cap\cdot)) \precsim \frac{1}{h^{-1}(1/t^2)\ln(h^{-1}(1/t^2))^{1 + \delta}} = \frac{1}{\beta^{-1}(1/t^2)\ln(\beta^{-1}(1/t^2))^{1 + \delta}},$$ for every boundary phase $\theta.$ Hence $$\dim_\mathcal{F}^+(\mu_\theta(I\cap\cdot)) \precsim \frac{1}{\beta^{-1}(1/t^2)\ln(\beta^{-1}(1/t^2))^{1 + \delta}}$$ as desired.

	Now we will prove that $\dim_\mathcal{F}^-(\mu_\theta(I\cap \cdot)) \succsim \beta^{-1}(1/t^2)^{\delta-1}$ for every $\delta > 0$ and every boundary phase $\theta.$ By Theorem \ref{SUBHDIM} it suffices to prove that for every $E\in I$ and every $\delta > 0$
	\begin{equation}
	\liminf_{L\to \infty} (\beta^{-1}(\norm{u_1}_L^2 \norm{u_2}_L^2))^{\delta - 1} \norm{u_1}_L^2 > 0.
	\end{equation}
	
	First, note that we have $|u_1(m)|^2 + |u_2(m)|^2 \geq \norm{\Phi_m(E)}^{-2}.$  We see, by (\ref{EXPBOUND}) and our choice of $L_n,$ for sufficiently large $n$ and $L_n \leq m < L_{n + 1},$ 
		\begin{align}
		\norm{u_1}_m^2 &> \frac12\left((L_n - L_{n-1})\left(\frac{L_{n-1}}{\beta^{-1}(L_{n - 1})}\right)^{-2\alpha(L_{n-1})} + l \left(\frac{L_{n}}{\beta^{-1}(L_{n})}\right)^{-2\alpha(L_{n})}\right)\\
		&\geq \frac{L_n}{\beta^{-1}(L_n)} + l \left(\frac{L_{n}}{\beta^{-1}(L_{n})}\right)^{-2\alpha(L_{n})}\label{U1Lower},
		\end{align}
	where $l = m - L_n + 1.$
	Similarly, we have 
		\begin{align}
		\norm{u_2}_m^2 &< L_n\left(L_{n-1}\beta^{-1}(L_{n - 1})\right)^{2\alpha(L_{n-1})} + l \left(L_n\beta^{-1}(L_n)\right)^{2\alpha (L_n)}\\
		&\leq L_n\beta^{-1}(L_n) + l \left(L_n\beta^{-1}(L_n)\right)^{2\alpha(L_n)}. \label{U2Upper}
		\end{align}
	For simplicity, let 
	\begin{align*}
	A_n &= \frac{L_n}{\beta^{-1}(L_n)} \\
	B_n &= \left(\frac{L_n}{\beta^{-1}(L_n)}\right)^{-2\alpha(L_n)} \\
	C_n &= L_n\beta^{-1}(L_n) \\
	D_n &= \left(L_n\beta^{-1}(L_n)\right)^{2\alpha (L_n)}
	\end{align*}
	so that (\ref{U1Lower}) becomes $\norm{u_1}_m^2 > A_n + l B_n$ and similarly (\ref{U2Upper}) becomes $\norm{u_2}_m^2 < C_n + l D_n.$
	By combining (\ref{U1Lower}), (\ref{U2Upper}) and (\ref{U1Condition}), and letting $1 \leq l < L_{n + 1} - L_n + 1,$ we obtain
				
		\begin{align}\label{DEFf}
		\begin{split}
		(\beta^{-1}(\norm{u_1}_L^2 \norm{u_2}_L^2))^{\delta - 1} &\norm{u_1}_L^2 \geq (\beta^{-1}((C_n + l D_n)(m \ln(m)^2)))^{\delta - 1} (A_n + l B_n) \\
		&= \frac{A_n + l B_n}{(\beta^{-1}((C_n + l D_n)(l + L_n - 1) \ln(l + L_n - 1)^2))^{1 - \delta}} \\
		&\equiv F_{n,\delta}(l).
		\end{split}
		\end{align}
		
	We now need to analyze lower bounds for $F_{n,\delta}(l)$ for $1 \leq l < L_{n+1} - L_n +1,$ so consider the two cases:
	
	{\bf Case 1:} $lB_n \leq A_n$
		
	{\bf Case 2:} $lB_n \geq A_n$
	
	We can see that case 1 is equivalent to the case where $l \leq D_n C_n L_n^{-2\epsilon}$ and case 2 is equivalent to the case where $l \geq D_n C_n L_n^{-2\epsilon}.$
	
	Considering case 1, we have:
	\begin{align}
	F_{n,\delta}(l) &\geq \frac{A_n}{(\beta^{-1}((C_n + C_nL_n^{-2\epsilon} D_n^2)(l + L_n - 1) \ln(l + L_n - 1)^2))^{1 - \delta}}\\
	&\geq \frac{A_n}{(\beta^{-1}((2C_nL_n^{-2\epsilon} D_n^2)(2D_nC_nL_n^{-2\epsilon}) \ln(2D_nC_nL_n^{-2\epsilon})))^{1 - \delta}}\\
	&= \frac{A_n}{(\beta^{-1}(4D_n^3C_n^2L_n^{-4\epsilon}\ln(2D_nC_nL_n^{-2\epsilon}))^{1 - \delta}}\\
	&\geq \frac{A_n}{(\beta^{-1}(4D_n^{3+1/2}C_n^{2+1/2}L_n^{-4\epsilon - 1/2}))^{1 - \delta}}.
	\end{align}
	
	We can now appeal to Lemma \ref{CONVEXITY} and the fact that $D_n = \beta(L_n)^{2(1 + \epsilon)} L_n^\epsilon$ to obtain:
	\begin{equation}
	\frac{A_n}{\beta^{-1}(4D_n^{3+1/2}C_n^{2+1/2}L_n^{-4\epsilon - 1/2}))^{1 - \delta}} \geq \frac{L_n^{1 - \epsilon}}{K L_n^{1 - \delta}},
	\end{equation}
	for some constant $K > 0.$ Since we may take $\epsilon$ arbitrarily small by taking $n$ sufficiently large, we conclude that this limits to $\infty$ for every $\delta > 0.$ 
	
	Now considering case 2, we have:
	\begin{align}
	F_{n,\delta}(l) &\geq \frac{A_n + l B_n}{(\beta^{-1}(KlD_n(2l) \ln(2l)^2))^{1 - \delta}}\\
	&\geq \frac{A_n + l B_n}{(\beta^{-1}(K) + 2\beta^{-1}(l) + \beta^{-1}(D_n) + \beta^{-1}(\ln(2l)^2))^{1 - \delta}}\\
	&\to \infty.
	\end{align}
	Thus for every $\delta > 0,$ $F_{n\delta}(l) \to \infty$ as $n, l\to\infty,$ which completes our proof.
	
	\end{proof}
		

\begin{proof}[Proof of Theorem \ref{SPARSEBARTHM}(iii)]
Once again, we will consider $\eta = 1.$ We will show that for Lebesgue a.e. $E,$ and a.e. $\theta,$ the equation $Hu = Eu$ has solutions with appropriate decay properties. 

Fix $\theta_0 = 0,$ and let $H = H_{\theta_0}.$ For each $m \in \Z^+,$ let $G_m$ be the operator on $l^2(\Z^+)$ given by 
\begin{equation}
\INNERPROD{\delta_i}{G_m\delta_j} = \delta_{i,m - 1}\delta_{j,m} + \delta_{i-1,m}\delta_{j, m-1}.
\end{equation}
For each $k \in Z^+,$ define new operators $H_k' = H - G_{L_k}$ and $\hat H_k = H - G_{L_k} - G_{L_k + 1},$ and for every $z \in \C,$ define the resolvent operators $G(z) = (H - z)^{-1}, G_k'(z) = (H_k' - z)^{-1},$ and $\hat G_k(z) = (\hat H_k - z)^{-1}.$ Moreover, for $i,j \in \Z^+,$ let the corresponding Green's functions be given by
\begin{align}
G(i,j,z) &= \INNERPROD{\delta_i}{G(z)\delta_j}\\
G_k'(i,j,z) &= \INNERPROD{\delta_i}{G_k'(z)\delta_j}\\
\hat G_k(i,j,z) &= \INNERPROD{\delta_i}{\hat G_k(z)\delta_j}.
\end{align}

Now considering some $n > L_k,$ we can use the resolvent identity $G(z) = G'_k(z) - G(z) G_{L_k} G'_k(z)$ to obtain
\begin{equation}
G(1,n,z) = - G(1, L_k - 1, z)G_k'(L_k, n, z).\label{eq:Greenfcn}
\end{equation}
A similar computation with $G'_k(z),$ yields the identity 
$$G'_k(z) = \hat G_k(z) - \hat G_k(z) G_{Lk + 1} G_k'(z) ,$$ 
so we have
\begin{equation}
G_k'(Lk, n ,z) = -\hat G_k(L_k, L_k, z) G_k'(L_k + 1, n, z) = \frac {-1}{V(L_k) - z} G_k'(L_k + 1, n , z).\label{eq:VarGreenfcn}
\end{equation}
Together, \eqref{eq:Greenfcn} and \eqref{eq:VarGreenfcn} yield
\begin{equation}
G(1,n,z) = G(1, L_k - 1, z) G_k'(L_k + 1, n , z) \frac 1 {V(L_k) - z}.
\end{equation}
Whenever $z = E + i\epsilon, \epsilon > 0,$ we see that $G(i,j,z)$ and $G'_k(i , j , z)$ have the form \eqref{eq:BORELTRANSFORM} and thus are Borel transforms of signed measures. We know (see e.g. \cite{SimonDim} for details) that Borel transforms have finite non-tangential limits a.e. on the real axis: $|G(i,j,E)| = |G(i,j, E + i0)| < + \infty$ and $|G'_k(i,j,E)| = |G'_k(i,j,E + i0)| < + \infty.$ 

Let us also recall Boole's equality for Borel transforms of singular measures: if $F(z)$ is the Borel transform of a singular measure on $\R$ such that $\mu(\R) = 1,$ then for any $\lambda > 0,$ we have $|\set{E: f(E) > \lambda}| = 2/\lambda.$ Since we have already shown that the spectral measures of $H$ for any vector $\delta_i$ are singular (Theorem \ref{SPARSEBARTHM} (ii) above), we conclude that
\begin{align}
\left|\set{E: |G(i,j,E)| > \lambda}\right| &\leq 4/\lambda,\\
\left|\set{E: |G_k'(i,j,E)| > \lambda}\right| &\leq 4/\lambda.
\end{align}
From this, we deduce that for any $j \geq 1, \gamma > 0,$ $k > 1, n > L_k$ and $E\in (-2,2),$ 
\begin{equation}
\left|\set{E: |G(1,n,E)| > \frac{f^j(L_k)^\gamma}{V(L_k) - 2}}\right| \leq \frac 8 {f^j(L_k)^{\gamma/2}},
\end{equation}
where $f(x) = \beta^{-1}(x)$ and $f^j(x)$ is the $j$-fold composition of $f$ with itself. Let us now fix $j \geq 1.$ By our choice of $L_k,$ $\sum_{k = 2}^\infty (f^j(L_k))^{-\gamma/2} < \infty$ for every $j \geq 1$ and $\gamma > 0.$ By the Borel-Cantelli lemma, for Lebesgue a.e. $E\in (-2,2),$ there exists a $K(E, j)$ such that for any $k > K(E, j)$ and $n = L_k + 1$ or $L_k + 2,$
\begin{equation}
|G(1,n,e)| \leq \frac{f^j(L_k)^\gamma}{V(L_k) - 2}.\label{eq:GreenBd}
\end{equation}

Now, if $E$ is such that the sequence $u_n = \set{G(1,n,E)}_{n = 1}^\infty$ exists, it necessarily solves the equation $Hu = Eu$ for $n > 2.$ Thus for any $L_k + 2 < n \leq L_{k + 1},$ we can recover $G(1,n,E)$ using $G(1,L_k + 1, E)$ and $G(1, L_k + 2, E)$ and the action of the free transfer matrix:
\begin{equation}
\begin{pmatrix}
G(1, n + 1, E)\\ 
G(1, n , E)
\end{pmatrix}
=
\Phi_{L_k + 2, n}(E) 
\begin{pmatrix}
G(1, L_k + 2, E) \\
G(1, L_k + 1, E)
\end{pmatrix}.
\end{equation}
Since we know that the free transfer matrix is bounded, we have $\norm{\Phi_{L_k + 2, n (E)}} \leq C(E)$ for $L_k + 2 < n < L_{k + 1}$ and $E\in (-2,2).$ Hence for $L_k < n \leq L_{k + 1}, $ \eqref{eq:GreenBd} holds for the same full measure set of $E$ as above and $k > K(E).$ 

It now follows that for Lebesgue a.e. $E \in (-2,2),$ there exists a solution $v$ of $Hu = Eu$ with $|v(0)|^2 + |v(1)|^2 = 1$ and a constant $C = C(E),$ such that for sufficiently large $k$ and $n > L_k,$
\begin{equation}
|v(n)| < C \frac{f^j(L_k)^\gamma}{V(L_k)} = C f^j(L_k)^\gamma \beta(L_k)^{-1} = Cf^j(L_k) L_{k + 1}^{-1/k}.
\end{equation}
Moreover, since there can be at most one subordinate solution of $Hu = Eu$ with the normalization property $|v(0)|^2 + |v(1)|^2 = 1$ which is decaying, $v$ must be the unique subordinate solution of $Hu = Eu.$ We also have, for $m \in \Z^+, L_n < m \leq L_{n + 1}$ with $n$ sufficiently large,
\begin{align}
\norm{v}_m^2 &= \sum_{j = 1}^m |v(j)|^2 \\
&= \sum_{j = 1}^{L_{k(E)}} |v(j)|^2 + \sum_{j = L_{k(E)} + 1}^m |v(j)|^2\\
&\leq C(E, v) + C\sum_{i = k(E)}^{n} L_{i + 1} f^j(L_i)^{2\gamma} L_{i + 1}^{-2/i}\\
&\leq C(E,v) + C f^j(L_n)^{2\gamma} L_{n + 1}.
\end{align}

Now we return to considering $H_\theta,$ where $\theta$ can vary. Recall that we can view $H_\theta$ as $H$ along with an appropriate rank-one perturbation at the origin. By the theory of rank-one perturbations (again, we refer readers to \cite{SimonDim} for full details), it is known that for any set $A\subset \R$ with $|A| = 0,$ we have $\mu(A) = 0$ for Lebesgue a.e. boundary phase $\theta.$ Since the set of energies for with the solution $v$ above does not exist is a Lebesgue null set, we can conclude that for a.e. boundary phase $\theta,$ the associate spectral measure $\mu$ is supported on the set of $E$ where the solution $v$ above exists. Furthermore, since $\mu$ must also be supported on the set of energies for which $u_1$ is subordinate, it follows that for a.e. $\theta$ and a.e. $E$ with respect to $\mu,$ $u_1$ must coincide with $v$ above.

For this $u_1$ and $m = L_n + L_n L_{n + 1}^{2/n},$ we have
\begin{align}
\norm{u_1}_m^2 &= \sum_{j = 1}^{L_n} |u_1(j)|^2 + \sum_{L_n + 1}^m |u_1(j)|^2\\
&\leq C(E,v) + C f^j(L_{n-1})^{2\gamma} L_{n} + C (m - L_n) f^j(L_n)^{2\gamma} L_{n + 1}^{-2/n} \\
& = C(E,v) + C f^j(L_{n-1})^{2\gamma} L_{n} + C f^j(L_n) L_n \\
&\leq C( 1 + f^j(L_n)^{2\gamma} L_n).
\end{align}

On the other hand, a similar analysis yields
\begin{equation}
\norm{u_2}_m^2 \geq L_n \beta(L_n)^2 \beta^{-1}(L_n)^{-1}
\end{equation}
for $m = L_n + L_n L_{n + 1}^{2/n}.$

Thus, if $g_k(x) = \beta^{-1}(1/x^2)f^j(x),$ and $m = L_n + L_n L_{n + 1}^{2/n},$ then
\begin{align}
g_k( \norm{u_1}_m^{-1} \norm{u_2}_m^{-1}) \norm{u_1}_m^2 &\leq \frac{C( 1 + f^j(L_n)^{2\gamma} L_n)}{\beta^{-1}(L_n \beta(L_n)^2 \beta^{-1}(L_n)^{-1})f^k(L_n \beta(L_n)^2 \beta^{-1}(L_n)^{-1})}\\
&\leq \frac{C( 1 + f^j(L_n)^{2\gamma} L_n)}{2L_n f^{k-1}(L_n)}.
\end{align}
Since this limits to 0 whenever $k \leq j,$ and since $j \geq 1$ was arbitrary, we conclude that $\dim^+_\mathcal{F}(\mu(A\cap \cdot)) \precsim \frac{1}{\beta^{-1}(1/t^2)}.$ 
\end{proof}

\section{Rank one perturbations: general results} \label{section:SULE}
We will now consider a probability measure $\mu$ on $\R$ and the self-adjoint operator $A: L^2(d\mu) \to L^2(d\mu)$ given by multiplication by $x.$ Let $\varphi$ be any cyclic unit vector in $L^2(d\mu).$ We define the rank one perturbation of $A$ by $\varphi$ as 
	\begin{equation} A_\lambda = A + \lambda \INNERPROD{\varphi}{\cdot}\varphi, \quad \lambda \in \R.\end{equation}
We will let $\mu_\lambda$ denote the spectral measure associated to $A_\lambda$ and $\varphi.$ Let $F_\lambda$ denote the Borel transform of $\mu_\lambda,$ and write $F_0 = F.$ Then
	\begin{align}
	F_\lambda(z) &= \frac{F(z)}{1 + \lambda F(z)} ,\\
	\IM F_\lambda(z) &= \frac{\IM F(z)}{|1 + \lambda F(z)|^2} ,\\
	d\mu_\lambda(x) &= \lim_{x\to\epsilon^+} \frac{1}{\pi} \IM F_\lambda(x+i\epsilon)dx,\\
	\mu_{\lambda,\text{sing}} &\text{ is supported by } \set{x: F(x+ i0) = -1/\lambda}.
	\end{align}

In addition to the Borel transform, we define 
	\begin{equation} G(x) = \int \frac{d\mu(y)}{(x-y)^2}. \end{equation}
It is well know that 
	\begin{equation} \set{x: G(x) < \infty, F(x+i0) = -\lambda^{-1}} = \text{set of eigenvalues of } A_\lambda. \end{equation}

\begin{mylemma} \label{IntLem} Let $\mathcal{F}$ be a family of comparable Hausdorff measure functions and let $f \in \mathcal{F}.$ Suppose that for a family of intervals $A_n,$ we have 
	$$|A_n| \leq f^{-1}(b_n)$$ 
where $b_n \geq 0$ is a summable sequence of real numbers.
Then $\dim_\mathcal{F}\left(\limsup A_n\right) \precsim f.$
\end{mylemma}
\begin{proof}
	Fix $g \in \mathcal{F}$ such that $f \prec g.$ That is, $\lim f(t)/g(t) = \infty,$ so $g(t) \leq f(t).$ Thus, $g^{-1}(t) \geq f^{-1}(t).$ Since $f$ is a Hausdorff dimension function, $f^{-1}(t) \to 0$ as $t\to 0.$ Hence $|A_n| \to 0,$ so given $\delta,$ we can choose $N_\delta$ so that $|A_n| \leq \delta$ for $n \geq N_\delta.$ Then for $m \geq N_\delta,$ $\bigcup_{n = m}^\infty A_n$ is a $\delta$-cover of $\limsup A_n.$ Thus,
		\begin{equation}
		\sum_{n = m}^\infty g(|A_n|) \leq \sum_{n = m}^\infty g(f^{-1}(b_n)) \leq \sum_{n = 1}^\infty g(g^{-1}(b_n)) < \infty.
		\end{equation}
	
	Thus, as $m \to \infty,$ we have $\sum_{n = m}^\infty g(|A_n|) \to 0.$
	We conclude that
		\begin{equation} 
		\lim_{\delta \to 0} \inf_{\delta\text{-covers}} \set{\sum_{i = 1}^\infty g(|F_i|)} = 0.
		\end{equation}
	Thus $\dim_\mathcal{F}(\limsup A_n) \prec g$ for every $f \prec g,$ so $\dim_\mathcal{F}(\limsup A_n) \precsim f.$
\end{proof}

\begin{mythm} \label{MyDelRio} 
Let $f(t)$ be a zero dimensional Hausdorff dimension function, let $\mathcal{F} = \set{f(t^\alpha): \alpha > 0},$ and suppose $d\mu(E) = \sum_{n = 1}^\infty a_n d\delta_{E_n}(E)$ where $a_n$ obeys the condition that 
	$$|a_n| \leq f^{-1}(b_n),$$ where $b_n$ is a summable sequence of positive real numbers. Then for every $\lambda$ we have $\dim_{\mathcal{F}}(d\mu_\lambda) \precsim f(t^2).$
\end{mythm}

\begin{proof}
	Let $G(x)$ be defined as above and let $S = \set{x: G(x) = \infty, x\not\in \set{E_i}_{i = 1}^\infty}.$ Then the Aronszajn-Donoghue theory \cite{SIMON} says that for any $\lambda \ne 0,$ $d\mu_\lambda^{sc}$ is supported by $S,$ Thus, the spectral measure $d\mu_\lambda$ is supported by $S\cup \set{\text{eigenvalues of} A_\lambda}.$ Since the set of eigenvalues is countable, it will not contribute to the dimension of $\supp(d\mu_\lambda),$ so it suffices to prove that $S$ has $\dim_\mathcal{G}(S) \precsim f(t^2).$ 

	Fix $\epsilon > 0,$ let $c_{n,\epsilon} = \frac 12 |a_n|^{1/2-\epsilon}$ and let $A_n^\epsilon = [E_n - c_{n,\epsilon}, E_n + c_{n,\epsilon}].$ Then 
		$$|A_n^\epsilon| = 2c_{n,\epsilon} = |a_n|^{1/2-\epsilon} \leq f^{-1}(b_n)^{1/2-\epsilon}.$$ 
	Now by Lemma \ref{IntLem}, for every $\epsilon > 0$ and every $f(t^{2/(1 - 2\epsilon)}) \prec g(t)$ we have $\mu^{g}(\limsup A_n^\epsilon) = 0.$

	


	Now it remains to show that $S \subset \limsup A_n^\epsilon$ for every $\epsilon.$ That is, it remains to show that if $x\not\in \limsup A_n^\epsilon$ and $x \not\in\set{E_n}_{n = 1}^\infty,$ then $G(x) < \infty.$ If $x \not\in \limsup A_n^\epsilon$ then for some $N_0,$ we must have $x\not\in \bigcup_{n = N_o}^\infty A_n^\epsilon.$ Now observe that 
		\begin{align}
		G(x) &= \sum_{n = 1}^\infty \frac{a_n}{|x - E_n|^2} \\
		&= \sum_{n = 1}^{N_0} \frac{a_n}{|x - E_n|^2} + \sum_{n = N_0}^{N\infty} \frac{a_n}{|x - E_n|^2} \\
		&\leq C + \sum_{n = N_0}^{\infty} \frac{a_n}{c_{n,\epsilon}^2} \\
		&\leq C + \sum_{n = N_0}^\infty 2 a_n^{2\epsilon}\\
		&\leq C + \sum_{n = N_0}^\infty 2f^{-1}(b_n)^{2\epsilon}.
		\end{align}
	The first sum is bounded because $x\not\in \set{E_n}_{n = 1}^\infty.$ Since $f$ is a zero dimensional Hausdorff dimension function, $t^\alpha \prec f^{-1}(t)$ for every $\alpha > 0,$ so $f^{-1}(b_n)^{2\epsilon}$ is summable, so we have $G(x) < \infty.$ Thus $S \subset \limsup A_n^\epsilon$ for every $\epsilon.$ Thus $\dim_\mathcal{F}(S) \precsim f(t^{2/(1 - 2\epsilon)})$ for every $\epsilon > 0.$ By our definition of $\mathcal{F},$ it follows that $\dim_\mathcal{F}(S) \precsim f(t^2).$
\end{proof}

By considering the larger family $\mathcal{G} = \set{f(t^\alpha)^\beta: \alpha,\beta} > 0,$ we can actually take $b_n = 1/n$ in the above theorem and conclude with the same result.

\begin{mydef} Let $H$ be a self-adjoint operator on $l^2(\Z^\nu).$ We say that $H$ has semi-uniformly localized eigenfunctions (SULE) if and only if $H$ has a complete set $\set{\varphi_n}_{n = 1}^\infty$ of orthnormal eigenfunctions, there is $\alpha > 0$ and $m_n \in \Z^\nu, n = 1,...,$ and for each $\delta > 0,$ a $C_\delta$ so that 
\begin{equation} |\varphi_n(m)| \leq C_\delta e^{\delta |m_n| - \alpha|m - m_n|} \end{equation}
for all $m\in \Z^\nu$ and $n = 1,2,....$
\end{mydef}

\begin{mylemma}[\cite{DelRioJitLastSim}] \label{DelRio3} Suppose that $H$ has SULE. Then there are $C$ and $D$ and a labeling of eigenfunctions so that \begin{equation} |\varphi_n(0)| \leq C\exp(-Dn^{1/\nu}).\end{equation}
\end{mylemma}

\begin{mythm} \label{MySULERes}
Suppose $H$ has SULE and let $\mathcal{F} = \set{\ln(1/t)^{-\alpha}: 0 < \alpha < \infty}.$ Let $H_\lambda = H + \lambda\INNERPROD{\delta_0}{\cdot}\delta_0.$ Let $d\mu$ be the spectral measure of $H$ associated to $\delta_0,$ and let $d\mu_\lambda$ be the corresponding spectral measures for $H_\lambda.$ Then for every $\lambda,$ $\dim_\mathcal{F}(\supp(d\mu_\lambda)) \precsim \ln(1/t)^{-\nu}.$
\end{mythm}

\begin{proof}
	Let $\mu$ be the spectral measure associated to $H$ and $\delta_0,$ and $\mu_\lambda$ the spectral measures of $H_\lambda.$ Observe that we have 
		\begin{equation} 
		\delta_0 = \sum_{n = 1}^\infty \varphi_n(0) \varphi_n.
		\end{equation}
	Set $a_n = \varphi_n(0).$ We can see that 
		\begin{equation}
		d\mu(E) = \sum_{n = 1}^\infty a_n d\delta_{E_n},
		\end{equation}
	where $E_n$ is the eigenvalue associated to the eigenfunction $\varphi_n.$ 
	By Lemma \ref{DelRio3}, we have $|a_n| \leq C\exp(-Dn^{1/\nu}) = C/f^{-1}(n).$ We can see that $f(n) = \left(\frac{-\ln(n)}{D}\right)^{\nu},$ so by Theorem \ref{MyDelRio} we conclude that, for every $\lambda,$ $\dim_\mathcal{F}^+(d\mu_\lambda) \precsim \left(-\ln(t)\right)^{-\nu}.$
\end{proof}

\section{Dynamical bounds} \label{section:QuantDynam}
Consider a separable Hilbert space $\mathscr{H}$ and $H: \mathscr{H} \to \mathscr{H}$ a self adjoint operator. Let us fix a vector $\psi\in \mathscr{H}$ with $\norm{\psi} = 1.$ The time evolution of $\psi$ is given by 
\begin{equation}
\psi(t) = e^{-iHt}\psi.
\end{equation}

We now introduce the following notation:
\begin{equation}
\left\langle A \right\rangle (t) = \INNERPROD{\psi(t)}{A\psi(t)}
\end{equation}
for any operator $A$ on $\mathscr{H},$ and 
\begin{equation}
\left\langle f \right\rangle_T = \left\langle f(t) \right\rangle_T = \frac{1}{T} \int_0^T f(t) dt
\end{equation}
for any measurable function $f.$

We also have the moments of the position operator in $l^2(\Z^\nu):$
\begin{equation}
|X|^m = \sum_{n \in \Z^\nu} |n|^m \INNERPROD{\delta_n}{\cdot}\delta_n.
\end{equation}

\begin{mydef} Let $\mu$ be a finite Borel measure, and let $\rho$ be a Hausdorff dimension function. We say the measure $\mu$ is uniformly $\rho$-H\"older continuous (U$\rho$H) if there exists a constant $C>0$ such that $\mu(I) < C\rho(|I|)$ for sufficiently small intervals $I.$
\end{mydef}

\begin{mydef} Let $H$ be a self-adjoint operator on a Hilbert space $\mathscr{H}.$ We denote the the $\rho$-continuous subspace as
\begin{equation}
\mathscr{H}_{\rho c} := \set{\psi \in \mathscr{H}: \mu_\psi \text{ is } \rho\text{-continuous}}.
\end{equation}
\end{mydef}

\begin{mythm}[Rogers and Taylor \cite{RT1}]\label{UHRTThm}
Let $\mu$ be a finite Borel measure on $\R$ and let $\rho$ be a Hausdorff dimension function. Then
$\mu$ is $\rho$-continuous if and only if for each $\epsilon > 0$ there are mutually singular Borel measures $\mu_1^\epsilon, \mu_2^\epsilon,$ such that $d\mu = d\mu_1^\epsilon + d\mu_2^\epsilon,$ $\mu_1^\epsilon$ is U$\rho$H, and $\mu_2^\epsilon(\R) < \epsilon.$
\end{mythm}

\begin{mythm}
Let $\rho$ be a Hausdorff dimension function and $\mathscr{H}_{uh}(\rho) = \set{\psi: \mu_\psi \text{ is U}\rho\text{H}}.$ Then $\mathscr{H}_{uh}(\rho)$ is a vector space and 
\begin{equation}
\overline{\mathscr{H}_{uh}(\rho)} = \mathscr{H}_{\rho c}.
\end{equation}
\end{mythm}

\begin{proof}
The only non-trivial vector space property is that $\mathscr{H}_{uh}(\rho)$ is closed under linear combinations, so that is all we will prove here. Let $\psi_1, \psi_2 \in \mathscr{H}_{uh}(\rho),$ and let $\varphi = a\psi_1 + b \psi_2.$ By assumption, there are constants $C_1$ and $C_2$ and $\delta > 0$ such that $\mu_{\psi_1}(I) < C_1\rho(|I|)$ and $\mu_{\psi_2}(I) < C_2 \rho(|I|)$ for all intervals $I$ with $|I| < \delta.$ For such $I,$ let $P_I$ denote the spectral projection on $I.$ Then 
\begin{align}
\begin{split}
\mu_\varphi(I) &= \INNERPROD{\varphi}{P_I \varphi} \\
&= \INNERPROD{a\psi_1 + b \psi_2}{aP_I \psi_1 + bP_I\psi_2} \\
&\leq |a|^2\INNERPROD{\psi_1}{P_I\psi_1} + |b|^2\INNERPROD{\psi_2}{P_i\psi_2} + 2|a||b||\INNERPROD{\psi_1}{P_I\psi_2}|.
\end{split}
\end{align}
Now
\begin{align}
\begin{split}
|\INNERPROD{\psi_1}{P_I\psi_2}| &\leq \sqrt{\INNERPROD{\psi_1}{P_I\psi_1}\INNERPROD{\psi_2}{P_I\psi_2}} \\
&\leq \frac 12 (\INNERPROD{\psi_1}{P_i\psi_1} + \INNERPROD{\psi_2}{P_I\psi_2}),
\end{split}
\end{align}
so we have
\begin{align}
\begin{split}
\mu_\varphi(I) &\leq |a|^2\INNERPROD{\psi_1}{P_I\psi_1} + |b|^2\INNERPROD{\psi_2}{P_i\psi_2} + 2|a||b||\INNERPROD{\psi_1}{P_I\psi_2}|\\
&\leq (|a|^2 + |a||b|)\INNERPROD{\psi_1}{P_I\psi_1} + (|b|^2 + |a||b|)\INNERPROD{\psi_2}{P_I\psi_2}\\
&= (|a|^2 + |a||b|)\mu_{\psi_1}(I) + (|b|^2 + |a||b|)\mu_{\psi_2}(I)\\
&\leq C_1(|a|^2 + |a||b|)\rho(|I|) + C_2(|b|^2 + |a||b|)\rho(|I|) \\
&= C \rho(|I|).
\end{split}
\end{align}
Thus $\mathscr{H}_{uh}(\rho)$ is a vector space.

Since Theorem \ref{UHRTThm} implies that $\mathscr{H}_{uh}(\rho) \subset \mathscr{H}_{\rho c},$ we have $\overline{\mathscr{H}_{uh}(\rho)} \subset \overline{\mathscr{H}_{\rho c}}.$ Since $\mathscr{H}_{\rho c}$ is closed, we have $\overline{\mathscr{H}_{uh}(\rho)} \subset {\mathscr{H}_{\rho c}}.$ By Theorem \ref{UHRTThm}, we can decompose $d\mu_\varphi, \varphi\in \mathscr{H}_{\rho c},$ into a sum of mutually singular measures:  $d\mu_\varphi = d\mu_1^\epsilon + d\mu_2^\epsilon,$ where $d\mu_1^\epsilon$ is U$\rho$H and $d\mu_2^\epsilon(\R) < \epsilon.$ Let $S_\epsilon$ be a Borel set that supports $\mu_2^\epsilon$ such that $\mu_1^\epsilon(S_\epsilon) = 0,$ and let $P_{S_\epsilon}$ denote the spectral projection on $S_\epsilon.$ We have
$$ \varphi = P_{S_\epsilon}\varphi + (1 - P_{S_\epsilon})\varphi$$
with $P_{S_\epsilon}\varphi \in \mathscr{H}_{uh}(\rho)$ and $\norm{(1 - P_{S_\epsilon})\varphi}^2 < \epsilon.$ Thus $\varphi$ is the norm-limit of vectors in $\mathscr{H}_{uh}(\rho),$ so $\mathscr{H}_{\rho c} \subset \overline{\mathscr{H}_{uh}(\rho)}$
\end{proof}

\begin{mylemma}\label{SOMETHING}
If $\mu_\psi$ is U$\rho$H, then there exists a constant $C = C(\psi)$ such that for any $\varphi\in \mathscr{H}$ with $\norm{\varphi}\leq 1,$ we have
\begin{equation}
\langle | \INNERPROD{\varphi}{\psi(t)}|^2\rangle_T < C \rho(1/T).
\end{equation}
\end{mylemma}

\begin{mythm}\label{DYNCOMPRES} Suppose $\mu_\psi$ is U$\rho$H. Then there exists a constant $C = C(\psi)$ such that for any compact operator $A, p \in \N,$ and $T> 0:$
\begin{equation}
\left\langle|\left\langle A \right\rangle|\right\rangle_T < C^{1/p} \norm{A}_p \rho(1/T)^{1/p}.
\end{equation}
\end{mythm}

\begin{proof}
Since $A$ is compact, the spectral theorem guarantees the existence of orthonormal bases $\set{\psi_n}_{n = 1}^\infty, \set{\varphi_n}_{n = 1}^\infty,$ and a monotonely decreasing sequence $\set{E_n}_{n = 1}^\infty,, E_n \geq 0,$ such that $A$ is given by the norm-convergent sum
\begin{equation}
A = \sum_{n = 1}^\infty E_n\INNERPROD{\varphi_n}{\cdot}\psi_n.
\end{equation}
Moreover, $\norm{A}_p = \norm{E_n}_{l^p}.$ Thus we have
\begin{align}
\langle |\langle A \rangle | \rangle_T &= \left\langle \left| \sum_{n = 1}^\infty E_n \INNERPROD{\varphi_n}{\psi(t)}\INNERPROD{\psi(t)}{\psi_n}\right|\right\rangle_T \\
&\leq \sum_{n = 1}^\infty E_n\langle|\INNERPROD{\varphi_n}{\psi(t)}\INNERPROD{\psi(t)}{\psi_n}|\rangle_T \\
&\leq \sum_{n = 1}^\infty E_n (\langle|\INNERPROD{\varphi_n}{\psi(t)}|^2\rangle_T)^{1/2} (\langle|\INNERPROD{\psi(t)}{\psi_n}|^2\rangle_T)^{1/2}.
\end{align}
If we let $p,q\in \N$ be such that $1/p + 1/q = 1,$ then we may apply H\"older's inequality to obtain
\begin{align}
\langle |\langle A \rangle | \rangle_T &\leq \norm{E_n}_{l^p} \norm{(\langle|\INNERPROD{\varphi_n}{\psi(t)}|^2\rangle_T)^{1/2} (\langle|\INNERPROD{\psi(t)}{\psi_n}|^2\rangle_T)^{1/2}}_{l^q}\\
&\leq \norm{A}_p \norm{\langle|\INNERPROD{\varphi_n}{\psi(t)}|^2\rangle_T}_{l^q}^{1/2} \norm{\langle|\INNERPROD{\psi(t)}{\psi_n}|^2\rangle_T}_{l^q}^{1/2}.\label{PNORMINEQ}
\end{align}
Moreover, by Lemma \ref{SOMETHING}, we have
\begin{align*}
\langle|\INNERPROD{\varphi_n}{\psi(t)}|^2\rangle_T &< C(\psi) \rho(1/T) \\
\langle|\INNERPROD{\psi(t)}{\psi_n}|^2\rangle_T &< C(\psi) \rho(1/T).
\end{align*}
Since the $\psi_n$ and $\varphi_n$ form orthonormal bases, and since $e^{-iHt}$ is unitary, we have
\begin{equation}
\sum_{n = 1}^\infty \langle | \INNERPROD{\varphi_n}{\psi(t)}|^2\rangle_T = \sum_{n = 1}^\infty \langle | \INNERPROD{\psi(t)}{\psi_n}|^2\rangle_T = \norm{\psi}^2 = 1.
\end{equation}
Thus
\begin{align}
\begin{split}\label{QNORMINEQ}
\norm{\langle | \INNERPROD{\varphi_n}{\psi(t)}|^2\rangle_T}_{l^q}^q &< (C(\psi) \rho(1/T))^{q - 1} \\
\norm{\langle | \INNERPROD{\psi(t)}{\psi_n}|^2\rangle_T}_{l^q}^q &< (C(\psi) \rho(1/T))^{q - 1}.
\end{split}
\end{align}
Putting (\ref{QNORMINEQ}) and (\ref{PNORMINEQ}) together, we have
\begin{equation}
\langle |\langle A \rangle | \rangle_T < \norm{A}_p (C(\psi) \rho(1/T))^{(q-1)/q} = C(\psi)^{1/p} \norm{A}_p \rho(1/T)^{1/p},
\end{equation}
which completes our proof.

\end{proof}

Now we can prove Theorem \ref{POSBDTHM}:

\begin{proof}[Proof of Theorem \ref{POSBDTHM}.]
Let $\psi_{\rho c} = P_{\rho c}\psi, \psi_{\rho s} = (1 - P_{\rho c})\psi.$ By Theorem \ref{UHRTThm}, there exist mutually singular Borel measures, $\mu_1, \mu_2$ such that $d\mu_{\psi_{\rho c}} = d\mu_1 + d\mu_2,$ where $\mu_1$ is U$\rho$H and $\mu_2(\R) < \frac 1 2 \norm{\psi_{\rho c}}^2.$ Let $S_1$ be a Borel set that supports $\mu_1$ and $\mu_2(S_1) = 0.$ Let $P_{S_1}$ denote the spectral projection on $S_1$ and set $\psi_1 = P_{S_1} \psi_{\rho c}$ and $\psi_2 = (1 - P_{S_1})\psi_{\rho c} + \psi_{\rho s}.$ Clearly $\psi = \psi_1 + \psi_2$ and $\psi_1 \perp \psi_2.$ Moreover, we have $d\mu_{\psi_1} = d\mu_1,$ so $\psi_1$ is U$\rho$H and
\begin{equation}
\norm{\psi_1}^2 = \int d\mu_{\psi_1} = \int d\mu_{\psi_{\rho c}} - \int d\mu_2 \geq \frac 12 \norm{\psi_{\rho c}}^2
\end{equation}
and
\begin{equation}
1 = \norm{\psi_1}^2 + \norm{\psi_2}^2.
\end{equation}

Let $P_N$ be the projection on the sphere of radius $N \in [0,\infty),$ defined by 
\begin{equation}
P_N = \sum_{|n| \leq N} \INNERPROD{\delta_n}{\cdot}\delta_n.
\end{equation}
We can see that 
\begin{align*}
Tr(P_N) &= \sum_{n \in \Z^\nu}\INNERPROD{\delta_n}{P_N\delta_n} \\
&= \sum_{n \in \Z^\nu}\INNERPROD{\delta_n}{\sum_{|k| \leq N} \INNERPROD{\delta_k}{\delta_n}\delta_k}\\
&= \sum_{n \in \Z^\nu}\sum_{|k| \leq N} \INNERPROD{\delta_n}{\delta_k}\INNERPROD{\delta_k}{\delta_n}\\
&= \sum_{|n| \leq N} 1\\
&= c_\nu N^\nu.
\end{align*}
Where $c_\nu$ depends only on the space dimension $\nu.$ Thus $P_N$ is compact and it follows from Theorem \ref{DYNCOMPRES} that there exists a constant $C_{1},$ which depends only on $\psi_1,$ such that for $T$ sufficiently large and $N > 0,$ 
\begin{align}
\begin{split}
\langle \norm{P_N\psi_1(t)}^2\rangle_T &= \langle \INNERPROD{\psi_1(t)}{P_N\psi_1(t)}\rangle_T \\
&< C_1 Tr(P_N) \rho(1/T) \\
&< c_\nu C_\psi N^\nu \rho(1/T).
\end{split}
\end{align}
Moreover, we have
\begin{align*}
\langle \norm{P_{N}\psi(t)}^2\rangle_T &\leq \langle (\norm{P_N\psi_1(t)} + \norm{P_N\psi_2(t)})^2\rangle_T \\
&\leq \langle( \norm{P_N\psi_1(t)} + \norm{\psi_2})^2\rangle_T \\
&\leq (\sqrt{\langle \norm{ P_N \psi_1(t)}^2\rangle_T} + \norm{\psi_2})^2.
\end{align*}
Now if we set $$N_T = \left(\frac{\norm{\psi_1}^4}{64C_1c_\nu\rho(1/T)}\right)^{1/\nu}$$ then we have 
\begin{align*}
\langle \norm{P_{N}\psi(t)}^2\rangle_T &< \left(\frac{\norm{\psi_1}^2}{8} + \norm{\psi_2}\right)^2\\
&= \frac{\norm{\psi_1}^4}{64} + \norm{\psi_2}^2 + \frac 1 4 \norm{\psi_1}^2\norm{\psi_2} \\
&< \norm{\psi_2}^2 + \frac 12 \norm{\psi_1}^2\\
&= 1 - \frac 12 \norm{\psi_1}^2.
\end{align*}
Since 
\begin{equation}
\langle \norm {P_{N_T}\psi(t)}^2\rangle_T + \langle\norm{(1 - P_{N_T})\psi(t)}^2\rangle_T = 1,
\end{equation}
we have
\begin{equation}
\langle\norm{(1 - P_{N_T})\psi(t)}^2\rangle_T > \frac 12 \norm{\psi_1}^2.
\end{equation}
Hence
\begin{align}
\langle\langle |X|^m\rangle\rangle_T &= \left\langle \INNERPROD{\psi(t)}{\sum_{n \in \Z^\nu}|n|^m \INNERPROD{\delta_n}{\psi(t)}\delta_n}\right\rangle_T \\
&\geq  \left\langle \INNERPROD{\psi(t)}{\sum_{|n| \geq N_T}N_T^m \INNERPROD{\delta_n}{\psi(t)}\delta_n}\right\rangle_T\\
&= N_T^m \left\langle\INNERPROD{\psi(t)}{(1-P_{N_T})\psi(t)}\right\rangle_T\\
&\geq \frac 12 \norm{\psi_1}^2 N_T^m\\
&= \frac 12 \norm{\psi_1}^2 \left(\frac{\norm{\psi_1}^4}{64 C_1c\nu}\right)^{m/\nu}\rho(1/T)^{-m/\nu}.
\end{align}
This completes our proof.

\end{proof}

\section*{Acknowledgement}
We would like to thank S. Jitomirskaya for presenting us with the problem that lead to this work, and for many useful suggestions and comments on earlier versions of the manuscript. We are also very grateful for W. Liu's numerous useful suggestions, discussions, and careful reading of earlier versions of the manuscript. This research was partially supported by NSF DMS-1901462, DMS-2052899, DMS-2000345, and DMS-2052572.


\bibliographystyle{abbrv} 
\bibliography{CurrentProgress22.bib}

\end{document}